\documentclass[reqno]{amsart}

\usepackage[latin1]{inputenc}

\usepackage{amssymb}
\usepackage{color}

\usepackage{graphicx}
\usepackage{epstopdf}

\newcommand{\N}{\mathbb{N}}
\newcommand{\Z}{\mathbb{Z}}

\newcommand{\R}{\mathbb{R}}
\newcommand{\C}{\mathbb{C}}

\newcommand{\om}{\omega}

\newcommand{\Pp}{\mathcal{P}}
\newcommand{\E}{\mathcal{E}}           
\newcommand{\Orth}{\mathcal{O}}

\DeclareMathOperator{\PD}{PD}
\DeclareMathOperator{\vol}{vol}

\DeclareMathOperator{\Diff}{Diff}

\newcommand{\CP}{\mathbb{CP}}
\newcommand{\w}{\mathsf{w}}
\newcommand{\lambdabar}{\overline{\lambda}}

\newcommand{\nblowup}[1]{\#\,#1\overline{\CP}\,\!^2}

\def\eoe{\unskip\ \hglue0mm\hfill$\between$\smallskip\goodbreak}
\def\cqfd{\unskip\ \hglue0mm\hfill$\Box$\smallskip\goodbreak}

\theoremstyle{plain}
\newtheorem{thm}{Theorem}[section]
\newtheorem*{thm*}{Theorem}
\newtheorem{prop}[thm]{Proposition}
\newtheorem*{prop*}{Proposition}
\newtheorem{lemma}[thm]{Lemma}
\newtheorem*{lemma*}{Lemma}
\newtheorem{cor}[thm]{Corollary}
\newtheorem*{cor*}{Corollary}

\theoremstyle{definition}
\newtheorem{defn}[thm]{Definition}
\newtheorem*{defn*}{Definition}
\newtheorem*{ackn}{Acknowledgments}

\theoremstyle{remark}
\newtheorem{remark}[thm]{Remark}
\newtheorem*{remark*}{Remark}

\newtheorem{example*}{Example}

\newtheorem*{facts*}{Facts}

\newcommand{\comment}[1]   {}
\newcommand{\mute}[2] {}
\newcommand{\printlabel}[1] {}
\newcommand{\printversion}{}

\begin{document}



\title{Packing numbers of rational ruled $4$-manifolds} 
\author{Olguta Buse}
\address{Department of Mathematical Sciences, IUPUI, Indianapolis, 
USA.}
\email{buse@math.iupui.edu}

\author{Martin Pinsonnault}
\address{Department of Mathematics, The University of Western Ontario, London, 
Canada.}
\email{mpinson@uwo.ca}
\thanks{Partially supported by NSERC grant RGPIN 371999.}
\printversion


\begin{abstract} We completely solve the symplectic packing problem with equally sized balls for any rational, ruled, symplectic $4$-manifolds. We give explicit formulae for the packing numbers, the generalized Gromov widths, the stability numbers, and the corresponding obstructing exceptional classes. As a corollary, we give explicit values for when an ellipsoid of type $E(a, b)$, with $\frac{b}{a} \in \N$, embeds in a polydisc $P(s,t)$. Under this integrality assumption, we also give an alternative proof of a recent result of M. Hutchings showing that the ECH capacities give sharp inequalities for embedding  ellipsoids into  polydisks.
\end{abstract}

\maketitle

\tableofcontents


\pagestyle{myheadings}
\markboth{Packing numbers}{Buse and Pinsonnault}


\section{Introduction and main results}

\subsection{Background}
Let $\sqcup_k B(c)$ be the disjoint union of $k$ standard $2n$-dimensional balls of radius $r$ and capacity $c=\pi r^{2}$. The $k^{\text{th}}$ packing number of a compact, $2n$-dimensional, symplectic manifold $(M, \omega)$ is
\[ 
p_k (M,\omega) = \frac{\sup_c \vol(\sqcup_k B(c)}{\vol(M,\omega)}
\]
where the supremum is taken over all $c$ for which there exists a symplectic embedding of $\sqcup_k B(c)$ into $(M,\omega)$. Naturally, $p_k(M,\omega) \leq 1$. When $p_k(M,\omega)=1$ we say that $(M,\omega)$ admits a full packing by $k$ balls, otherwise we say that there is a packing obstruction. An essentially equivalent invariant is the generalized $k^{\text{th}}$ Gromov width $\w_k(M,\omega)$ defined by setting
\[
\w_k(M,\omega) = \sup_{c>0} \{ c ~|~\sqcup_k B(c) \text{~embeds symplectically into~} (M,\omega) \}
\]
For a compact manifold of dimension $2n$ the width $\w_{k}$ is thus bounded by
\[0<\w_k(M^{2n},\omega)\leq c_{\vol}(M^{2n},\omega):= \sqrt{\frac{n!\vol\left(M^{2n},\omega\right)}{k}}\]
Although no general tools are known to compute those invariants for arbitrary symplectic manifolds, some results can be derived from complex algebraic geometry. For instance, in \cite{MP}, D. McDuff and L. Polterovich computed $p_{k}(\CP^{2})$, for $k\leq 9$. They also
proved that $p_{k}(\CP^{n}) = 1$ whenever $k=p^{n}$ and that $\lim_{k\to\infty} p_k(M,\omega) = 1$ for any compact symplectic manifold. In view of that later result, it is natural to ask whether the sequence $p_k(M,\omega)$ is eventually stable, that is, whether there is a number $N_{{\rm stab}}(M,\omega)$ such that $p_k(M,\omega)= 1$ for all $k \geq N_{{\rm stab}}(M,\omega)$. To date, this remains an interesting open question (see \cite{CHLS} and~\cite{Bi-SPAG} for a complete discussion). The only general result in that regard is due to
P. Biran (\cite{Bi}, \cite{Bi2}) who settled this question positively for all closed symplectic $4$-manifolds whose symplectic forms (after rescaling) are in rational cohomology classes. His techniques allowed him to obtain some lower and upper bounds for
$N_{{\rm stab}}(M^{4},\omega)$ which can be explicitly computed in some cases. In particular, he showed that $N_{{\rm stab}}(\CP^{2})\leq 9$ which, in view of McDuff and Polterovich results, is sharp.

The same techniques apply to rational ruled symplectic $4$-manifolds. Recall that, after rescaling, any such manifold is symplectomorphic to either
\begin{itemize}\label{Normalization}
\item the trivial bundle $M_{\mu}^{0} := (S^2\times S^2, \om^0_{\mu})$, where the symplectic area of the a section $S^2\times\{*\}$ is $\mu\geq 1$ and the area of a fiber $\{*\}\times S^{2}$ is 1; or
\item the non trivial bundle $M_{\mu}^{1} := (S^{2}\ltimes S^{2},\om^{1}_{\mu})$, where the symplectic area of a section of self-intersection $-1$ is $\mu>0$ and the area of a fiber is $1$.
\end{itemize}
In~\cite{Bi} Biran showed that 
\begin{equation}\label{birdiotriv}
p_k(M_{\mu}^0) =\min\left\{ 1,\frac{k}{2 \mu}\inf\left(\frac{\mu n_1 + n_2}{2 n_1 +2 n_2 -1}\right)^{2}\right\}
\end{equation}
where the infimum is taken over all naturals $n_1, n_2$ for which the Diophantine equations
\begin{equation}\label{birdioeq}
2 n_1 n_2 =\left(\sum_{i=1}^{i=k}m_i^2 \right)-1~, \qquad 2 (n_1 + n_2) =\left(\sum_{i=1}^{i=k}m_i\right) +1
\end{equation}
admit a vector solution $(m_1,\ldots, m_{k})\in \N^{k}$, while
\[
p_k(M_{\mu}^1) =
\min\left\{ 1,\frac{k}{2\mu+1}\inf\left(\frac{\mu n_1 + n_2}{n_1 +2 n_2 -1}\right)^{2}\right\}
\]
where the infimum is taken over all naturals $n_1, n_2$ for which the equations
\[n_1(2n_{2}- n_1) =\left(\sum_{i=1}^{i=k}m_i^2\right) -1 \qquad n_1 + 2n_2 =\left(\sum_{i=1}^{i=k}m_i\right) +1
\]
admit a vector solution $(m_1,\ldots, m_{k})\in\N^{k}$ (as we will see later, those equations simply ensure that the exceptional classes in the $k$-fold blow-up of $M_{\mu}^{i}$ have nonnegative symplectic areas). He also obtained the following bounds for the stability number of $M_{\mu}^{0}$:
\[2 \mu \leq  N_{{\rm stab}}(M_{\mu}^{0})\leq 8 \mu\]

Building on~\cite{Bi,MP}, F. Schlenk~\cite{Sh2} later computed the packing numbers $p_{k}(M_{\mu}^{i})$, $i=0,1$, for $k\leq 7$ (those can be found in Appendix~A), and proved that
\[
\max(8,2 \mu+1) \leq  N_{{\rm stab}}(M_{\mu}^{1}) \leq
\begin{cases}
8 \mu+4 & \text{~if~} \frac{1}{2}\leq \mu \\
\frac{2\mu+1}{\mu^{2}} & \text{~if~} \mu<\frac{1}{2}
\end{cases}
\]

The above results reduce, in principle, the computation of the packing numbers $p_{k}(M_{\mu}^{i})$, $k\geq 8$, and of the stability numbers $N_{{\rm stab}}(M_{\mu}^{i})$ to purely arithmetic problems. However, since their general solutions are not known, they do not yield explicit formulae in terms of the parameters $k$ and~$\mu$.

\subsection{Main results}

In this paper, we use a modified version of Li-Li's reduction algorithm \cite{Li-Li-SymplecticGenus},\cite{Li-Li-Diff}, to compute the packing numbers, the generalized Gromov widths, and the stability numbers of rational ruled symplectic 4-manifolds. We also identify the exceptional homology classes that give the obstructions to symplectic embeddings of $k$ balls in $M_{\mu}^{i}$, for $k\geq 8$. We observe that our method can be used, in principle, to compute the packing numbers of any $k$-fold symplectic blow-up of $\CP^2$. We also note that D. McDuff and F. Schlenk used a similar method in~\cite{MS} to fully describe the embedding functions of four dimensional ellipsoids into standard balls.

\subsubsection{The Trivial bundle}
For the trivial bundle $M_{\mu}^{0}$, our computations of the generalized Gromov widths $\w_{k}(M_{\mu}^{0})$ reveal that the obstructions to the embeddings of $k\geq 8$ balls in $M_{\mu}^{i}$ depend in an essential way on the parity of $k$. Indeed, fixing $k\geq 8$ and viewing $\w_{k}=\w_{k}(\mu)$ as a function of $\mu\geq 1$, we show that there are only finitely many obstructions for $k$ odd, while there are infinitely many obstructions for $k$ even. 
\begin{thm}\label{thm:MainTheoremTrivial}
Let $M^0_{\mu}=(S^{2}\times S^{2},\mu\sigma\oplus\sigma)$ with $\mu\geq 1$.
\begin{enumerate}
\item When $k=2p+1$ is odd, the generalized Gromov width $\w_{2p+1}(M_{\mu}^{0})$ is given by
\[
\w_{2p+1}(M_{\mu}^{0}) =
\begin{cases}
c_{\vol}=\sqrt{\frac{ 2 \mu}{2p+1}} & \text{~if~}\mu\in\left[ 1,~ p+1-\sqrt{2p+1} \right)   \\
\frac{\mu+p}{2p+1}         & \text{~if~}\mu\in\left[ p+1-\sqrt{2p+1},~ p+1 \right) \\
1                          & \text{~if~}\mu\in\left[ p+1,~ \infty \right)
\end{cases}
\]
\item When $k=2p$ is even, there exist a decreasing sequence $\{\delta_n\}$ with limit $\lambda = \frac{p-2 + \sqrt{p^2-4p}}{2} $ and intervals $I_n$ given by $I_{0}= [p,~\infty)$, $I_n =[\delta_n,~\delta_{n-1})$, and $I_{\infty} =[1,~\gamma)$, as well as a sequence of linear functions $w_{n}:\R\to\R$, $n\geq 1$, such~that
\[
\w_{2p}(M_{\mu}^{0}) =
\begin{cases}
c_{\vol}=\sqrt{\frac{\mu}{p}}  & \text{~if~}\mu\in I_{\infty}\\

 w_{n}(\mu) & \text{~if~}\mu\in I_n ,~ n \geq 1\\

1 & \text{~if~}\mu\in I_{0}
\end{cases}
\]
\end{enumerate}
\end{thm}

In Section~\ref{section:TrivialBundle}, Proposition~\ref{thm:MainTheoremTrivialExplicitEven}
gives explicit formulae for the functions $w_{n}$, as well as complete descriptions of the even generalized Gromov widths $\w_{2p}$ as piecewise linear functions of~$\mu$. As an immediate corollary, we get the packing numbers of $M^0_{\mu}$  (see~Corollaries~\ref{cor:MainPackingNumbersTrivialOdd} and~\ref{cor:MainPackingNumbersTrivialEven}) and we compute the stability numbers of $M_{\mu}^{0}$, namely

\begin{cor}\label{cor:MainStabilityNumbersTrivial}
The odd stability number of $M_{\mu}^{0}$ is 
\[
N_{\text{odd}}(M_{\mu}^{0})=
\begin{cases}
7 &\text{~if~}\mu=\frac{8}{7}\\
9 &\text{~if~}\mu\in\left[1,~\frac{8}{7}\right)\cup\left(\frac{8}{7},~2\right]\\
2\left\lceil\mu+\sqrt{2\mu}\right\rceil + 1 
					&\text{~if~}\mu\in\left(2,~\infty\right]
\end{cases}
\]
while its even stability number is given by 
\[N_{\text{even}}(M_{\mu}^{0})=2\left\lceil \mu + 2 + \frac{1}{\mu}\right\rceil\] 
It follows that 
\[N_{\text{stab}}(M_{\mu}^{0})=
\begin{cases}
9 & \text{~if~} \mu=\frac{8}{7}\\
N_{\text{odd}} & \text{~if~} \mu\in\left[1,~\frac{8}{7}\right)\cup\left(\frac{8}{7},~2\right]\\
N_{\text{odd}}(M_{\mu}^{0})-1 & \text{~if~}\mu\in\left(2,~\infty\right]\end{cases}
\]
\end{cor}

Before we move on  to describe our results in the twisted case, let us explain the following consequence of Theorem~\ref{thm:MainTheoremTrivial}. Given positive real numbers $a,b,s,t$, recall that the standard $4$-dimensional ellipsoid $E(a,b)$ is defined by setting
\[
  E(a,b) :=
  \left\{ z \in \C^2 ~|~ \frac{\pi|z_1|^2}{a} + \frac{\pi |z_2|^2}{b} \leq 1\right\}
\]
while the standard $4$-dimensional polydisk $P(s,t)$ is given by
\[
  P(s ,t) :=
  \left\{ z \in \C^2 ~|~ \pi|z_1|^2 \leq s,~ \pi|z_2|^2\leq t \right\}
\]
Recently D. Muller~\cite{Mu} showed that the problem of embedding an ellipsoid into a polydisc is equivalent with embedding a collection of balls of various sizes into the polydisc. Using this, we can state
\begin{cor}\label{muler} 
Let $k$ be any integer greater than $8$ and let $a,s,t$ be any positive real numbers with $s < t.$  Denote by $\mu = \frac{a}{s}$. The following are equivalent:
\renewcommand{\labelenumi}{\roman{enumi})}
\renewcommand{\theenumi}{\roman{enumi})}
\begin{enumerate}
  \item $E(a, k a) \hookrightarrow P(s,t)$ 
  \item If $k =2p+1 $ then $\frac{a}{s} \leq \w_k =\min\{1, c_{\vol},
    \frac{\mu + p}{2p+1} \}.$ 
 If $k=2p$  then $\frac {a}{s} \leq \w_{k}
    =\min_{n \in \N} \{1, c_{\vol}, w_n \}.$ 
Moreover, there is an $n \in \N \cup \{\infty \}$ so that $\mu=\frac{a}{s} \in I_n$ and the precise value of this $\w_k$ is given by  Theorem \ref{thm:MainTheoremTrivial} part (ii).
\end{enumerate}
\end{cor}
The proof of that corollary is given at the end of the paper. In that section, we also give an alternative proof, valid only for the case when $b/a$ is an integer, of a recent result of M. Hutchings~\cite{H2}, \cite{H4} stating that the embedded contact homology (ECH) capacities give sharp conditions under which an ellipsoid of type $E(a, b)$ embeds in a polydisk $P(\nu,\mu).$

\subsubsection{The Nontrivial bundle}

One would expect results similar to those of Theorem~\ref{thm:MainTheoremTrivial} to hold for the twisted bundle $M_{\mu}^{1}$. However, it turns out there is no essential difference between odd and even widths. Instead, all the complexity appears at the special value $k=8$.

\printlabel{thm:MainTheoremTwistedK=8}
\begin{thm}\label{thm:MainTheoremTwistedK=8}
There exist three functions $u_{1}(\mu,n)$, $u_{2}(\mu,n)$, $u_{3}(\mu,n)$, depending on $\mu\in(0,\infty)$ and $n\in\Z\setminus{0}$, all linear in $\mu$, such~that
\[
\w_{8}\left(M_{\mu}^{1}\right) = \min_{n\in\Z\setminus\{0\}} 
\left\{
c_{\vol}=\sqrt{\frac{2\mu+1}{8}},~u_{1}(\mu,n),~u_{2}(\mu, n),~u_{3}(\mu,n),~\frac{6\mu+6}{17}
\right\}
\]
\end{thm}
In Section~\ref{section:TwistedBundle}, we give explicit formulae for the functions $u_{i}(\mu,n)$, as well as complete descriptions of the generalized Gromov width $\w_{8}$ as a piecewise linear functions of~$\mu$. As a corollary, we show that there exist infinitely values of $\mu$ for which we can fully pack the nontrivial bundle with $8$ disjoint balls. Indeed, if we define the set $\mathcal{S}\subset(0,\infty)$ by setting
\[
\mathcal{S} = \left\{\frac{8n^{2}-8n+1}{16n^{2}},~\frac{1}{2},~\frac{8n^{2} + 8n + 1}{16n^{2}}\right\},~n\geq 1
\]
then we have
\begin{cor}
There is a full packing of the nontrivial bundle $M_{\mu}^{1}$ by $8$  balls if, and only if, $\mu\in\mathcal{S}$.
\end{cor}
The general case $k\geq 9$ is easier to deal with as the number of obstructions is always finite, namely
\printlabel{thm:MainTheoremTwisted}
\begin{thm}\label{thm:MainTheoremTwisted} 
Given $k\geq 9$, let us write $k=2p$ or $k=2p+1$ depending on the
parity of $k$, and let $\mu \in (1/2,\infty)$. Then the 
$k^{\text{th}}$ generalized Gromov width of $M_{\mu}^{1}$ is given by:
\[
  \w_{2p}(M_{\mu}^{1}) =
  \begin{cases}
    c_{\vol}=\sqrt{\frac{2\mu+1}{2p}} & \text{~if~} \mu\in \left[1/2,~ p-\sqrt{2p}\right)\\
    \frac{p+\mu}{2p} & \text{~if~} \mu\in\left[p-\sqrt{2p}~ , p\right) \\
    1 & \text{~if~} \mu\in\left[p,~\infty\right)\\
  \end{cases}
\]
and
\[
  \w_{2p+1}(M_{\mu}^{1}) =
  \begin{cases}
    c_{\vol}=\sqrt{\frac{2\mu+1}{2p+1}} & \text{~if~} \mu\in \left[1/2,~ \frac{p^{3}-2p^{2}+1-(p-1)\sqrt{2p+1}}{p^{2}}\right)\\
    \frac{p(p+\mu-1)}{2p^{2}-p-1} & \text{~if~} \mu\in \left[\frac{p^{3}-2p^{2}+1-(p-1)\sqrt{2p+1}}{p^{2}},~ \frac{p(p-1)}{p+1}\right)\\
    \frac{p+\mu}{2p} & \text{~if~} \mu\in\left[\frac{p(p-1)}{p+1} ,~ p\right) \\
    1 & \text{~if~} \mu\in\left[p,~\infty\right)\\
  \end{cases}
\] 
\end{thm}
The previous results allow us to compute the stability numbers of the nontrivial bundle, namely
\printlabel{cor:MainPackingNumbersTwisted}
\begin{cor}\label{cor:MainPackingNumbersTwisted}
The stability number $N_{{\rm stab}}$ of $M_{\mu}^{1}$ is
\[
N_{{\rm stab}}=
\begin{cases}
8 & \text{~if~} \mu\in  \mathcal{S} \\
9 & \text{~if~} \mu\in   \left(0,~\frac{3}{2}\right) \\
N_{{\rm even}}-1  & \text{~if~} \mu\in \left[\frac{3}{2},~\infty \right)
\end{cases}
\]
\end{cor}
To conclude this introduction, we want to point out that our presentation is written with a dual purpose in mind. First, we obtained our results through an hybrid process of mathematical reasoning and computer-aided symbolic computations using SAGE~\cite{sage}, and our exposition replicates part of that process. Secondly, although we don't discuss the arithmetic aspects of the reduction process in the present paper, we advance the idea that the reduction algorithm is an effective computational tool that can be used in many other instances where one must deal with Diophantine approximation problems involving hyperbolic lattices.
\begin{ackn}
  The authors would like to thank Yael Karshon and Dusa McDuff for
  their interest in this work and for many useful comments and
  suggestions. Many thanks to Michael Hutchings for providing us with
  an early version of \cite{H2} and sharing with us his computations
  for the ECH capacities of polydiscs. We thank  MSRI where
  part of this work was completed.
\end{ackn}

\section{Embedding balls in symplectic $4$-dimensional rational manifolds}

\subsection{Symplectic embeddings of balls and symplectic blow-ups}
Using the correspondence between ball embeddings and blow-ups, the problem of deciding whether a collection $\mathcal{B}=\sqcup_{i}B(\delta_{j})$ of $k$ disjoint balls of capacities $\delta_{j}$ embeds symplectically in $M_{\mu}^{i}$ reduces to the question of understanding the symplectic cone of the $k$-fold blow-up of $M_{\mu}^{i}$ (see, for instance, McDuff-Polterovich~\cite{MP}). Because the $k$-fold blow-up of $M_{\mu}^{i}$ is diffeomorphic to $\CP^{2}$ blown-up $(k+1)$ times, this is in turn equivalent to understanding the symplectic cone of the rational surfaces $\CP^{2}\nblowup{(k+1)}$, for $k\geq 1$.

\subsection{Reduced classes and symplectic cones of rational surfaces}
Given $n\geq 1$, let us denote by $X_{n}:=(X_{n},\om_{\lambda;\delta_{1},\ldots,\delta_{n}})$ the $n$-fold symplectic blow-up of $(\CP^{2},\om_{\lambda})$ at $n$ disjoint balls of capacities $\delta_{1},\dots\delta_{n}$. Let $\{L,E_{1},\cdots,E_{n}\}$ be the standard basis of $H_{2}(X_{n};\Z)$ consisting of the class of a line $L$, and the classes $E_{i}$, $1\leq i\leq n$, of the exceptional divisors. Using Poincaré duality, the cohomology class of the symplectic form on $X_{n}$ is identified with $\lambda L-\sum_{i}\delta_{i}E_{i}$ while, given any compatible almost-complex structure $J$ on $X_{n}$, the first Chern class
$c_{1}:=c_{1}(J)\in H^{2}(X_{n};\Z)$ is identified with the homology class $K:=3L-\sum_{i} E_{i}$.

The intersection product gives $H_{2}(X_{n};\Z)$ the structure of an odd unimodular lattice of type $(1,n)$, while $H_{2}(X_{n};\R)$ becomes an inner product space of signature $(1,n)$. Let $\Pp$ and $\Pp_{+}$ denote, respectively, the positive cone and the forward cone in
$H_{2}(X_{n};\R)$:
\[
  \Pp :=
  \left\{ A\in H_{2}(X_{n};\R)~|~ A\neq 0\text{,~and~} A\cdot A\geq 0 \right\},
\]
\[
  \Pp_{+} :=
  \left\{a_{0}L-\sum_{i} a_{i}E_{i}\in\Pp~|~ a_{0}\geq 0 \right\}.
\]
Let $\mathcal{C}_{K}\subset H_{2}(X_{n};\R)$ be the $K$-symplectic
cone, that is,
\[
  \mathcal{C}_{K} =
  \{A\in H_{2}(X_{n};\Z)~|~A=\PD[\om]\text{~for some~}\om\in\Omega_{K}\},
\]
where $\Omega_{K}$ is the set of orientation-compatible symplectic forms with $K$ as the symplectic canonical class. Similarly, let $\E_{K}\subset H_{2}(X_{n};\Z)$ be the set of symplectic exceptional homology classes, that is,
\[
  \E_{K} :=
  \left\{E~|~E\cdot E=-1\text{~and $E$ is represented by some embedded
                          $\om$-symplectic sphere, }\om\in\Omega_{K}\right\}.
\]
Building on the work of Taubes on Seiberg-Witten and Gromov invariants, T.-J. Li and A.-K. Liu characterized the symplectic cone of smooth, closed, oriented 4-manifolds with $b^{+}=1$ in terms of exceptional classes. In the case of $X_{n}$, this gives
\printlabel{thm:CharacterizationEK}
\begin{thm}[see \cite{Li-Liu-Uniqueness}, Theorem 3]
 \label{thm:CharacterizationEK}
  \[
    \mathcal{C}_{K} =
    \{A\in\Pp_{+}~|~A\cdot E >0\text{~for all~}E\in\E_{K}\}.
  \]
\end{thm}
Since $\E_{K}$ is not explicitly known for $n\geq 10$, this characterization cannot be used directly to show that a given class $A\in\Pp_{+}$ belongs to $\mathcal{C}_{K}$. However, the group $\Diff_{+}$ of orientation preserving diffeomorphisms acts on $H_{2}(X_{n};\Z)$, and any diffeomorphism preserving $K$ also preserves the sets $\E_{K}$ and $\mathcal{C}_{K}$. Let us write $\Orth(1,n)$ for the group of orthogonal transformations of $H_{2}(X_{n};\Z)$, $D(1,n)$ for the image of $\Diff_{+}$ in $\Orth(1,n)$, and $D_{K}(1,n)$ for the subgroup of $D(1,n)$ fixing $K$. Recall that if a class $A\in H_{2}(X_{n};\Z)$ of self-intersection $\pm1$ or $\pm2$ is represented by a smooth embedded sphere, then the reflection about $A$
\[r_{A}(B):=B - 2\left(\frac{A\cdot B}{A\cdot A}\right)A \]
belongs to $D(1,n)$. Assume $n\geq 3$ and set
\begin{equation}\label{def:BasicRoots}
\begin{aligned}
  \alpha_{0}&=L-E_{1}-E_{2}-E_{3}\\
  \alpha_{i}&=E_{i}-E_{i+1}\,, \quad 1\leq i\leq n-1.
\end{aligned}
\end{equation}
For $n\geq 3$, those classes are represented by smooth embedded spheres, and since $\alpha_{i}\cdot\alpha_{i} = -2$ and $K\cdot\alpha_{i}=0$, the reflections $r_{\alpha_{i}}$ belong to $D_{K}(1,n)$. The reflexion $\mathfrak{C}:=r_{\alpha_{0}}$, classically known as the \emph{Cremona} transformation, takes a class $\left(a_{0}\,; a_{1},\ldots,a_{n}\right)$ to the class 
\[\left(a_{0}-d\,; a_{1}-d,a_{2}-d, a_{2}-d, a_{3}-d,a_{4}\ldots,a_{n}\right)\]
where $d=a_{1}+a_{2}+a_{3}-a_{0}$, while the reflexion $r_{\alpha_{i}}$, $i\geq 1$, permutes the coefficients $a_{i}$ and~$a_{i+1}$.

The key ingredient to understand the action of $D_{K}(1,n)$ on the symplectic cone $\mathcal{C}_{K}$ is the notion of a \emph{reduced class}:
\printlabel{def:ReducedClasses}
\begin{defn}\label{def:ReducedClasses}
Let $k\geq 3$. A class $A=a_{0}L-\sum_{i} a_{i}E_{i}$ is said to be reduced with respect to the basis $\{L, E_{1},\ldots,E_{k}\}$ if $a_{1}\geq a_{2}\geq\cdots\geq a_{k}\geq 0$ and $a_{0} \geq a_{1}+a_{2}+a_{3}$.
\end{defn}
\printlabel{thm:Diff+ReducedClasses}
\begin{thm}\label{thm:Diff+ReducedClasses} Assume $n\geq 3$.
\begin{enumerate}
  \item {\rm(}see{\rm~\cite{Li-Li-Diff}, Theorem 3.1)} The group $D(1,n)$
    is generated by the reflections $\{r_{L}, r_{E_{1}},
    r_{\alpha_{0}},\ldots,r_{\alpha_{n}}\}$. In particular, it follows
    that the group $D_{K}(1,n)$ is generated by the reflections
    $\{r_{\alpha_{0}},\ldots,r_{\alpha_{n}}\}$.
\item {\rm(}see~{\rm\cite{Li-Li-SymplecticGenus}, Theorem D and
  \cite{Li-Liu-Uniqueness}, Theorem 1)} The group $D_{K}(1,n)$ acts
  transitively on $\E_{K}$.
\item {\rm(}see~{\rm\cite{Gao}, Proposition 2.2)} The orbit of an element
  $A\in\Pp$ under the action of $D(1,n)$ contains a unique reduced
  class.
\item {\rm(}see~{\rm\cite{Li-Liu-Uniqueness}, Proposition 4.9 (3))} A
  reduced class $A=a_{0}L-\sum_{i}a_{i}E_{i}$ belongs to
  $\mathcal{C}_{K}$ if and only if $a_{i}>0$ for all $i$.
\end{enumerate}
\end{thm}

Combining Theorem~\ref{thm:CharacterizationEK} with Theorem~\ref{thm:Diff+ReducedClasses} (2), we have
\printlabel{cor:FundamentalDomain}
\begin{cor}\label{cor:FundamentalDomain}
Let $C_{0}$ denote the set of reduced classes
$a_{0}L-\sum_{i}a_{i}E_{i}$ with $a_{i}>0$, $\forall i$.
\begin{enumerate}
\item A class $A\in\Pp_{+}$ belongs to $\mathcal{C}_{K}$ if and only
  if its orbit under $D_{K}(1,n)$ only contains classes
  $A=a_{0}L-\sum_{i}a_{i}E_{i}$ with $a_{i}>0$, $\forall i$.
\item The set $C_{0}$ is a fundamental domain of $\mathcal{C}_{K}$
  under the action of $D_{K}(1,n)$. In particular,
\[\mathcal{C}_{K} = D_{K}(1,n)\cdot C_{0}\]
\end{enumerate}
\end{cor}

\subsection{The reduction algorithm}
Theorem~\ref{thm:Diff+ReducedClasses} and Corollary~\ref{cor:FundamentalDomain} lead to a simple algorithm to decide whether a given class in $\Pp_{+}$ belongs to $\mathcal{C}_{K}$. To simplify notation, let us write $(a_{0}\,;a_{1},\ldots,a_{n})$ for the class
$a_{0}L-\sum_{i}a_{i}E_{i}$.
\begin{itemize}\setlength{\itemsep}{1mm}
\item[Step 1.] Set $\ell=-1$ and pick
  $v_{0}:=(a_{0}^{0}\,;a_{1}^{0},\ldots,a_{n}^{0})\in \Pp_{+}$.
\item[Step 2.] Increment $\ell$ by one. If $a_{1}^{\ell}\geq\cdots\geq
  a_{n}^{\ell}$, set $\hat{v}_{\ell}=v_{\ell}$ and go to Step
  3. Otherwise, using reflections $r_{\alpha_{i}}\in D_{K}(1,n)$ about
  $\alpha_{i}=E_{i}-E_{i+1}$, $1\leq i\leq n-1$, permute the coefficients of $v_{\ell}$
  so that $a_{1}^{\ell}\geq a_{2}^{\ell}\geq\cdots\geq a_{n}^{\ell}$ and write $\hat{v}_{\ell}$ for the reordered vector.
\item[Step 3.] Let $a_{n}^{\ell}$ be the last coefficient in
  $\hat{v}_{\ell}$. If $a_{n}^{\ell}<0$, then $v_{\ell}\not\in
  \mathcal{C}_{K}$. Thus $v_{0}\not\in\mathcal{C}_{K}$ and the
  algorithm stops.
\item[Step 4.] Let
  $d_{\ell}:=a_{1}^{\ell}+a_{2}^{\ell}+a_{3}^{\ell}-a_{0}^{\ell}$. If
  $d_{\ell}\leq 0$, then the class $\hat{v}_{\ell}$ is reduced. In
  that case
  \smallskip
  \begin{itemize}\setlength{\itemsep}{1mm}
    \item If $a_{n}^{\ell}>0$, then
      $\hat{v}_{\ell}\in\mathcal{C}_{K}$, hence
      $v_{0}\in\mathcal{C}_{K}$ as well, and the algorithm stops.
    \item If $a_{n}^{\ell}=0$, then $\hat{v}_{\ell}$ is in the
      boundary of $\mathcal{C}_{K}$, hence $v_{0}$ is in the boundary
      of $\mathcal{C}_{K}$ as well, and the algorithm stops.
  \end{itemize}
\item[Step 5.] The class $\hat{v}_{\ell}$ has non-negative coefficients
  but is not reduced. Apply the reflexion $r_{\alpha_{0}}\in
  D_{K}(1,n)$ about $\alpha_{0}=L-E_{1}-E_{2}-E_{3}$ to obtain the
  class vector
  \begin{align*}
    v_{\ell+1}&=(a_{0}^{\ell}-d_{\ell}\,;a_{1}^{\ell}-d_{\ell}, a_{2}^{\ell}-d_{\ell}, a_{3}^{\ell}-d_{\ell}, a_{4}^{\ell},\ldots,a_{n}^{\ell})\\
    &=(2a_{0}^{\ell}-a_{1}^{\ell}-a_{2}^{\ell}-a_{3}^{\ell};a_{0}^{\ell}-a_{2}^{\ell}-a_{3}^{\ell}, a_{0}^{\ell}-a_{1}^{\ell}-a_{3}^{\ell}, a_{0}^{\ell}-a_{1}^{\ell}-a_{2}^{\ell},a_{4}^{\ell},\ldots,a_{n}^{\ell})\\
    &=(a_{0}^{\ell+1};a_{1}^{\ell+1},\ldots,a_{n}^{\ell+1})
  \end{align*}
  and go back to Step 2.
\end{itemize}
We claim that the algorithm stops after finitely many iterations. To see this, first note that the self-intersection of all the vectors $\hat{v}_{\ell}$ is constant (since every $\hat{v}_{\ell}$ is obtained from $v_{0}$ by applying an element of $D_{K}(1,n)$). Because
$v_{0}\in\Pp_{+}$, this implies that
\[
  \hat{v}_{\ell}\cdot \hat{v}_{\ell} =
  (a_{0}^{\ell})^{2}-\sum_{i}(a_{i}^{\ell})^{2} \geq 0,
\]
so that $|a_{0}^{\ell}|\geq|a_{i}^{\ell}|$, for all $1\leq i\leq n$. Note also that since $a_{0}^{0}\geq 0$, we must have $a_{0}^{\ell}\geq 0$, for all $\ell\geq 0$. This follows from the fact that in Step (4) above, $a_{0}^{\ell+1}=2a_{0}^{\ell}-a_{1}^{\ell}-a_{2}^{\ell}-a_{3}^{\ell}<a_{0}^{\ell}$ is negative if and only if $2a_{0}^{\ell}<a_{1}^{\ell}+a_{2}^{\ell}+a_{3}^{\ell}$. But since
\[
  0 \leq
  A \cdot A =
  (a_{0}^{\ell})^{2}-\sum_{i}(a_{i}^{\ell})^{2}\quad\text{and}\quad (a_{1}^{\ell}+a_{2}^{\ell}+a_{3}^{\ell})^{2}\leq 3\left((a_{1}^{\ell})^{2}+(a_{2}^{\ell})^{2}+(a_{3}^{\ell})^{2}\right),
\]
that would imply 
\[
  (a_{1}^{\ell})^{2} + (a_{2}^{\ell})^{2} + (a_{3}^{\ell})^{2} \leq
  (a_{0}^{\ell})^{2} \leq
  \tfrac{3}{4}\left( (a_{1}^{\ell})^{2}+(a_{2}^{\ell})^{2}+(a_{3}^{\ell})^{2} \right).
\]
Hence $a_{0}^{\ell+1}$ must be non-negative. 

Since the algorithm stops whenever the smallest coefficient of some $\hat{v}_{\ell}$ is negative, let us suppose that, starting with some $v_{0}\in\Pp_{+}$, we obtain an infinite sequence of vectors $\hat{v}_{\ell}$ with non-negative coefficients $a_{0}^{\ell}\geq a_{1}^{\ell}\geq\cdots\geq a_{n}^{\ell}\geq 0$. By Corollary~\ref{cor:FundamentalDomain}, all the vectors $\hat{v}_{\ell}$ are in the closure of $\mathcal{C}_{K}$. Note also that none of the vectors $\hat{v}_{\ell}$ is reduced, so that $d_{\ell}>0$, for all $\ell\geq 0$.

Suppose that $v_{0}\in\mathcal{C}_{K}$, that is, all the coefficients $a_{i}^{\ell}$ are strictly positive. In that case, there exists a symplectic form $\omega_{0}$ such that $[\omega_{0}]=v_{0}$. The process produces a sequence $\omega_{\ell}$ of symplectic forms which are, by construction, diffeomorphic to the initial symplectic form $\omega_{0}$ and such that, for at least one index $i\geq 1$, $a_{i}^{(\ell)}:=\left< [\omega_{\ell}] , E_{i}\right>$ is a strictly decreasing sequence of positive numbers. By Taubes results on the equivalence of Seiberg-Witten and Gromov invariants, the Gromov invariant of $E_{i}$ only depends on the underlying smooth structure of $X_{n}$. Thus, the class $E_{i}$ contains an embedded $\omega_{\ell}$-symplectic sphere whose size is $a_{i}^{\ell}$. The diffeomorphism from $(X_{n},\omega_{\ell})$ to $(X_{n},\omega_{0})$ carries this sphere to an embedded symplectic sphere in $(X_{n},\omega_{0})$ of size $a_{i}^{\ell}$. However, the set of symplectic areas of exceptional spheres does not have accumulation points, see for instance~\cite{KKP} Lemma~4.1. Thus, the sequence $a_{i}^{\ell}$, being a decreasing sequence of positive numbers that does not have accumulation points, must be finite.

If $v_{0}$ is in the boundary of $\mathcal{C}_{K}$, then all $\hat{v}_{\ell}$ belong to $\partial\mathcal{C}_{K}$ and there are integers $N\geq 0$ and $2\leq m\leq n-1$ such that, for all $\ell\geq N$, we have $a_{m}^{\ell}\neq 0$ and
\[
  \hat{v}_{\ell} =
  (a_{0}^{\ell}\,;a_{1}^{\ell}, a_{2}^{\ell},\ldots, a_{m}^{\ell},0,\ldots,0)
\]
Because $\hat{v}_{\ell}\cdot\hat{v}_{\ell}\geq 0$ and $d_{\ell}>0$, we must have $m\geq 2$. If $m=2$, then $a_{3}^{N}=0$, so that $a_{3}^{\ell+1}=a_{2}^{\ell}-d_{\ell} = -d_{\ell}<0$, which contradicts the fact that $a_{i}^{\ell}\geq 0$, for all $\ell\geq
0$. Hence, $m\geq 3$. The vectors $(a_{0}^{\ell}\,;a_{1}^{\ell},\ldots,a_{m}^{\ell})$, $\ell>N$, have strictly positive coefficients and thus represent classes in $H_{2}(X_{m};\Z)$ which, by Corollary~\ref{cor:FundamentalDomain}, must be in $\mathcal{C}_{K}(X_{m})$. By the previous argument, the sequence $\hat{v}_{\ell}$ must be finite.

\subsection{Strategy for the computation of $\w_{k}(M_{\mu}^{i})$ and $p_{k}(M_{\mu}^{i})$}\label{sec:strategy}
Recall that any rational ruled $4$-manifold is, after rescaling, symplectomorphic to either
\begin{itemize}
  \item the trivial bundle $M_{\mu}^{0} := (S^2\times S^2,
    \om^0_{\mu})$, where the symplectic area of the a section
    $S^2\times\{*\}$ is $\mu\geq 1$ and the area of a fiber
    $\{*\}\times S^{2}$ is 1; or
  \item the nontrivial bundle $M_{\mu}^{1} := (S^{2}\ltimes
    S^{2},\om^{1}_{\mu})$, where the symplectic area of a section of
    self-intersection $-1$ is $\mu>0$ and the area of a fiber is $1$.
\end{itemize}

We identify the $k$-fold blow-up of $M^0_{\mu}$ of equal sizes $c=c_{1}=\cdots = c_{k}$, with $\CP^{2}$ blown up $(k+1)$ times endowed with a symplectic form $\om_{\mu,c}^{0}$ which gives area $\mu+1-c$ to a line and areas $\{\mu-c, 1-c, c,\ldots, c\}$ to the $(k+1)$ exceptional divisors. We represent the Poincaré dual of $[\om_{\mu,c}^{0}]$ by the vector
\begin{equation}\label{eq:InitialVectorTrivial}
  v_{\mu,c}^{0} :=
  \left(\mu+1-c\,; \mu-c, c^{\times (k-1)}, 1-c\right)\in H_{2}(X_{k+1};\R)
\end{equation}

Similarly, we identify the $k$-fold blow-ups of $M^1_{\mu}$ of equal sizes $c=c_{1}=\cdots = c_{k}$, with $\CP^{2}$ blown up $(k+1)$ times endowed with a symplectic form $\om_{\mu,c}^{1}$ which gives area $\mu+1$ to a line and areas $\{\mu, c,\ldots, c\}$ to the $(k+1)$
exceptional divisors. The Poincaré dual of $[\om_{\mu,c}^{1}]$ is thus represented by the vector
\begin{equation}\label{eq:InitialVectorTwisted}
  v_{\mu,c}^{1} :=
  \left(\mu+1\,; \mu, c^{\times k}\right)\in H_{2}(X_{k+1};\R)
\end{equation}

The existence of an embedding of $k$ disjoint balls of capacity $c$ into $M^i_{\mu}$ is then equivalent to $v_{\mu,c}^{i}$ belonging to the symplectic cone $\mathcal{C}_{K}$ of $X_{k+1}$, where $K:=(3\,;1^{\times (k+1)})$. Hence, given $k\geq 8$, the computation of the generalized Gromov widths $\w_{k}(M^0_{\mu})$ and of the packing numbers $p_{k}(M^0_{\mu})$ reduces to finding the largest capacity $c>0$ such that
\begin{itemize}
  \item $v_{\mu,c}^{i}\in\mathcal{P}_{+}$ (i.e.,
    $v_{\mu,c}^{i}\cdot v_{\mu,c}^{i}\geq 0$);
  \item the orbit of $v_{\mu,c}^{i}$ under $D_{K}(1,k+1)$ only
    contains non-negative vectors or, equivalently, the reduction
    algorithm applied to $v_{\mu,c}^{i}$ produces a reduced and
    non-negative vector.
\end{itemize}

A posteriori, once one knows the generalized Gromov widths given in Theorem~\ref{thm:MainTheoremTrivial} and Theorem~\ref{thm:MainTheoremTwisted}, one can easily check that those numbers are the right ones. Indeed, given $\mu$ and $\w_{k}=\w_{k}(\mu)$, it is enough to find two automorphisms $\phi_{1},
\phi_{2}\in D_{K}(1,k+1)$ such that
\begin{itemize}
  \item $\phi_{1}(v_{\mu,\w_{k}})$ is non-negative and reduced
  \item for each $\epsilon>0$, $\phi_{2}(v_{\mu,\w_{k}+\epsilon})$
    contains a negative coefficient.
\end{itemize}
This can be done using the reduction algorithm. Our strategy is then to proceed backward: (i)~use the algorithm to find upper bounds $w_{n}=w_{n}(\mu)$ for the value of $\w_{k}$, (ii)~find the smallest one, say $w_{n_{0}}$, and (iii)~show that one gets a nonnegative reduced vector after setting  $c=w_{n_{0}}$ in $v_{\mu,c}^{i}$. Since the algorithm consists in applying elements of $D_{K}(1,k+1)$ to the initial vector $v_{\mu,c}^{i}$, a simple dualization gives us an exceptional class $E\in\mathcal{C}_{K}(X_{k+1})$ that defines the obstruction. More precisely, if $\phi\in D_{k}$ is the automorphism corresponding to the upper bound $\w_{k}=w_{n_{0}}$, then
\[0=\left(\phi\,v_{\mu,\w_{k}}, E_{k+1}\right)=\left(v_{\mu,\w_{k}}, \phi^{*}\,E_{k+1}\right)\]
so that $\phi^{*}\,E_{k+1}$ is an obstructing exceptional class. Such a class is generally not unique and, in fact, the reduction process often gives finitely many choices.

\section{Embeddings of $k\geq 8$ disjoint balls in the trivial bundle $M_{\mu}^{0}$}\label{section:TrivialBundle}

This section will be dedicated to proving the results in Theorem~\ref{thm:MainTheoremTrivial} as well as introducing several immediate corollaries. As explained in Section~\ref{sec:strategy}, the computation of the generalized Gromov widths $\w_{k}(M^0_{\mu})$ and of the packing numbers $p_{k}(M^0_{\mu})$ reduces to finding the largest capacity $c>0$ such that the vector
\[v_{0}:=v_{\mu,c}^{0}=\left(\mu+1-c\,; \mu-c, c^{\times (k-1)}, 1-c\right)\] 
belongs to the closure of the symplectic cone of $X_{k+1}$.

The initial step of the reduction algorithm already gives nontrivial results. Indeed, the vector $v_{0}$ is non-negative only if $c\leq 1$, which is equivalent to the fact that the Gromov width is $\w_{1}(M^0_{\mu})=1$. This obviously gives an upper bound for $\w_{k}$, and we note that this bound is stronger than the volume condition whenever $\mu\geq \frac{k}{2}$. Now, the vector $v_{0}$ is ordered only if we suppose $c\geq 1/2$. When $0<c\leq\frac{1}{2}$, the reordered vector $\hat{v}_{0}$ is
\[
  \hat{v}_{0} =
  \left(\mu+1-c\,; \mu-c, 1-c, c^{\times (k-1)}\right),
\]
which is positive and reduced (since its defect is zero). Now, we observe that $c_{\vol}\leq 1/2$ if and only if $k\geq8\mu$. Hence, $\w_{k}(M_{\mu}^{0}) = c_{\vol}$ whenever $\mu\leq \frac{k}{8}$. That implies we have full packing by $k$ balls for all $\mu\in[1, k/8]$ while, for $\mu\in(k/8, \infty)$, we have the lower and upper bounds $1/2<\w_{k}\leq\min\{1,c_{\vol}\}$.

Given $n\in\N$, let us write 
\[\lambda_{n} = \mu+1-c-nd_{0}\]
where $d_{0}=2c-1$. Using this notation, we have
\[v_{0} = \left( \lambda_{0}\,;\lambda_{0}-1, c^{k-1}, (1-c) \right).\]
Given $\mu>k/8$ and $c\in(1/2,1]$, the vector $v_{0}$ is ordered, and has defect $d_{0}=2c-1>0$. Applying a Cremona transformation $\mathfrak{C}$ followed by the permutation $\mathfrak{R}:=(1, 2, k+1, k+2, 3, \ldots, k)$, we get the vector
\[
  v_{1} =
  \left(\lambda_{1}\,; \lambda_{1}-1, c^{\times(k-3)}, (1-c)^{\times 3}\right),
\]
which is ordered if and only if $\lambda_{1}-1 = \mu+1-3c \geq c$. Assuming $v_{1}$ ordered, its defect is also $d_{0}=2c-1>0$, so that we can apply another Cremona move $\mathfrak{C}$ and a reordering $\mathfrak{R}$ to get
\[v_{2} =
  \left( \lambda_{2}\,; \lambda_{2}-1, c^{\times(k-5)}, (1-c)^{\times 5} \right)
\]
provided $k\geq 5$. Clearly, we can repeat this process $n$ times to get a vector
\[v_{n} := (\mathfrak{R} \mathfrak{C})^{n}v_{0} =
  \left( \lambda_{n}\,; \lambda_{n}-1, c^{\times(k-2n)}, (1-c)^{\times (2n+1)} \right),
\]
as long as $2n\leq k$ and $\lambda_{n-1}-1=\mu+1-c-(n-1)d_{0}\geq c$.
\begin{lemma}\label{lemma:TrivialBundleEvenOddDichotomy}
Given $k\geq 4$, let us write $k=2p$ or $k=2p+1$ depending on the parity of $k$. Choose any $\mu\geq1$ and $c\in(1/2, c_{\vol}]$. Then the following holds:
\begin{itemize}
  \item If $k$ is even, the vector $v_{p-3}$ is ordered.
  \item If $k$ is odd, the vector $v_{p-2}$ is ordered.
\end{itemize}
\end{lemma}
\begin{proof}
Let $k=2p$. Then the vector $v_{p-3}$ is ordered if and only if $\lambda_{p-3}-1\geq c$, which is equivalent to
\[\frac{\mu+p-3}{2(p-2)} \geq c.\]
Clearly, it is sufficient to consider the case $c=c_{\vol}$ for which the previous inequality becomes
\[\frac{\mu+p-3}{2(p-2)}\geq \sqrt{\frac{\mu}{p}}.\]
This is equivalent to $f(\mu) = p(\mu+p-3)^{2}-4\mu(p-2)^{2}\geq 0$, which is a quadratic polynomial in $\mu$ with positive leading coefficient, whose roots are
\[\frac{p^{2}-5p+8 \pm 4(p-2)\sqrt{4-p}}{p}.\] 
So, for $p\geq 5$, there are no real roots, while in the case $p=4$, we have a double root at $\mu=1$.

\medskip
Similarly, for $k=2p+1$, the vector $v_{p-2}$ is ordered if and only if 
\[\frac{\mu+p-2}{2(p-1)} \geq c.\]
For $c = c_{\vol}$, we get
\[\frac{\mu+p-2}{2(p-1)}\geq \sqrt{\frac{2\mu}{2p+1}},\]
which is equivalent to $f(\mu) = (2p+1)(\mu+p-2)^{2}-8\mu(p-1)^{2}\geq 0$. Again, this is a quadratic polynomial in $\mu$ with positive leading coefficient, whose roots are
\[\frac{2p^{2}-5p+6 \pm (p-1)\sqrt{8(4-p)}}{2p+1}\] 
So, for $p\geq 5$, there are no real roots while, in the case $p=4$, we have a double root at $\mu=2$.
\end{proof}
\begin{remark}\label{rmk:AutomorphismsTrivialInitialStep}
Observe that the vector $v_{i+1}$ is obtained from $v_{i}$ by applying a Cremona transformation followed by the permutation $\mathfrak{R}:=(1, 2, k+1, k+2, 3, \ldots, k)$. For convenience, we denote the corresponding element of $D_{K}(1, k+1)$ by $\mathfrak{R} \mathfrak{C}$. We can thus write $v_{i}=( \mathfrak{R} \mathfrak{C} )^{i}v_{0}$.
\end{remark}

We now differentiate our discussion depending on whether $k$ isodd or even.

\subsection{The odd case $k=2p+1$}
This subsection will provide the proof of part $(i)$ of Theorem~\ref{thm:MainTheoremTrivial}.

By Lemma~\ref{lemma:TrivialBundleEvenOddDichotomy}, we know that the vector
\[v_{p-2} =
  \left(\lambda_{p-2}\,; \lambda_{p-2}-1,c^{\times 4},(1-c)^{\times (2p-3)} \right)
\]
is ordered whenever $k\geq 4$, $\mu\geq 1$, and $c\in(1/2, c_{\vol}]$. So, we can perform a $\mathfrak{R} \mathfrak{C}$ move to get
\[
  v_{p-1} =
  \left( \lambda_{p-1}\,;\lambda_{p-1}-1, c^{\times 2}, (1-c)^{\times (2p-1)} \right)
\]
We now consider two cases depending on whether $v_{p-1}$ is ordered or not.
\medskip
{\it Case 1}: Suppose $v_{p-1}$ is not ordered, that is,
$\lambda_{p-1}-1<c$. Since $\lambda_{p-2}-1\geq c$, we must have
\[
  (\lambda_{p-1}-1) - (1-c) =
  \lambda_{p-2}-(2c-1)-2+c =
  (\lambda_{p-2}-1)-c\geq 0,
\]
so that
\[
  \hat{v}_{p-1} =
  \left(\lambda_{p-1}\,; c^{\times 2}, \lambda_{p-1}-1, (1-c)^{\times(2p-1)} \right)
\]
is ordered with defect $d_{p-1}=2c-1>0$. Performing another $\mathfrak{R} \mathfrak{C}$ move gives
\[
  \hat{v}_{p} =
  \left(\lambda_{p}\,; (1-c)^{\times (2p+1)}, \lambda_{p}-1\right).
\]
Since $d_{p}=3(1-c)-\lambda_{p}=3(1-c)-\lambda_{p-1}+2c-1=(1-c)-(\lambda_{p-1}-1)
\leq 0$, the vector $\hat{v}_{p}$ is reduced. It is positive if and only if $\lambda_{p}-1\geq 0$. Solving for $c$ in the equation $\lambda_{p}-1\geq0$, we obtain a new upper bound for the generalized width $\w_{2p+1}$, namely
\[\w_{2p+1}\leq\frac{\mu+p}{2p+1}.\]
We note that this bound is stronger than the previous bound $\w_{2p+1}\leq 1$ only when $\mu\leq p+1$, while it is stronger than the volume condition whenever $\mu\in [\alpha_{-},\alpha_{+}]$, where $\alpha_{\pm} = p+1\pm\sqrt{2p+1}$.

\medskip
{\it Case 2}: Suppose that $v_{p-1}$ is ordered. Performing a $\mathfrak{R} \mathfrak{C}$
move yields
\[\hat{v}_{p} = \left(\lambda_{p}\,;\lambda_{p}-1, (1-c)^{\times(2p+1)}\right)\]
with defect
\[d_{p}=2(1-c)-1=1-2c<0.\]
Therefore, $\hat{v}_{p}$ is reduced and positive, so that there is an embedding of $2p+1$ balls of size $c$ into $M_{\mu}^{0}$. That occurs unless $\lambda_{p-1}-1<c$.  Solving for $c$ in the equation $\lambda_{p-1}-1\leq c$, we get a new lower bound for $\w_{2p+1}$,
namely
\[ \frac{\mu+p-1}{2p}\leq \w_{2p+1}.\]

\medskip
We therefore have:
\printlabel{prop:TrivialBundleOddCase}
\begin{prop}\label{prop:TrivialBundleOddCase}           
Let $k=2p+1\geq 9$ and consider $\mu\geq 1$. The $(2p+1)^{\text{th}}$
generalized Gromov width of $M_{\mu}^{0}$ is
\[
  \w_{2p+1}(M_{\mu}^{0})=
  \begin{cases}
    c_{\vol}            & \text{~if~}\mu\in
                           \left[ 1,~ p+1-\sqrt{2p+1} \right) \\
    \frac{\mu+p}{2p+1} & \text{~if~}\mu\in
                           \left[ p+1-\sqrt{2p+1},~ p+1 \right) \\
    1                  & \text{~if~}\mu\in  \left[ p+1,~ \infty \right)
  \end{cases}
\]
  \end{prop}             
\begin{proof} 
The previous discussion shows that $\w_{2p+1}(M_{\mu}^{0}) =
\min\left\{c_{\vol}, 1, \frac{\mu+p}{2p+1}\right\}$.
\end{proof}
An immediate consequence is the following
\begin{cor}\label{cor:MainPackingNumbersTrivialOdd}
Let $k=2p+1\geq 9$ and consider $\mu\geq 1$. The $(2p+1)^{\text{th}}$
packing number of $M_{\mu}^{0}$ is
\[
  p_{2p+1}(M_{\mu}^{0}) =
  \begin{cases}
    1                             & \text{~if~}\mu\in
                                      \left[ 1,~ p+1-\sqrt{2p+1} \right) \\
    \frac{(\mu+p)^{2}}{2\mu(2p+1)} & \text{~if~}\mu\in
                                      \left[ p+1-\sqrt{2p+1},~ p+1 \right) \\
    \frac{2p+1}{2\mu}             & \text{~if~}\mu\in
                                      \left[ p+1,~ \infty \right)
  \end{cases}
\]
\end{cor}
\printlabel{cor:OddStabilityNumbersTrivial}
\begin{cor}\label{cor:OddStabilityNumbersTrivial}
The odd stability number of $M_{\mu}^{0}$ is 
\[
N_{{\rm odd}}(M_{\mu}^{0})=
\begin{cases}
7 &\text{~if~}\mu=\frac{8}{7}\\
9 &\text{~if~}\mu\in\left[1,~\frac{8}{7}\right)\cup\left(\frac{8}{7},~2\right]\\
2\left\lceil\mu+\sqrt{2\mu}\right\rceil + 1 
					&\text{~if~}\mu\in\left(2,~\infty\right]
\end{cases}
\]
\end{cor}
\begin{proof}
By Proposition~\ref{prop:PackingNumbersProduct<8}, the only pair $(\mu,k)$ for which we have full packing by $k=2p+1\leq 7$ balls is $(8/7,~7)$. On the other hand, the largest root of the polynomial in $p$
\[
(2p+1)(\mu+p)^{2}-2\mu(2p+1)^{2}
\]
obtained by setting
\[c_{\vol}^{2}=\frac{\mu+p}{2p+1}\]
is
\[r(\mu) =  \mu+\sqrt{2\mu}\]
The integer $J(\mu) := \max \{9,~ 2\lceil r(\mu) \rceil + 1\}$ gives the odd stability number in the range $k\geq 9$. Since $J(8/7)=9$, the result follows.
\end{proof}
As explained in Section~\ref{sec:strategy}, we can combine the above results with Remark~\ref{rmk:AutomorphismsTrivialInitialStep} to find obstructing exceptional classes in $H_{2}(X_{2p+2};\Z)$. These results can be easily translated into curves in the $2p+1$-fold blow-up of $M_{\mu}^{0}$ as well by using the identification of the two spaces.

\printlabel{cor:TrivialOddCaseObstructionClasses}
\begin{cor}\label{cor:TrivialOddCaseObstructionClasses}
The exceptional classes in $H_{2}(X_{2p+2};\Z)$ that give the obstructions to the embedding of $2p+1$ balls into $M_{\mu}^{0}$ are of type
\begin{align*}
  \left( 1\,; 1^{\times 2}, 0^{\times (2p-1)}\right) 
  	& \text{~when~} \mu\in\left[ p+1,~ \infty \right);\\
  \left( p\,; p-1, 1^{\times 2p}, 0\right) 
  	& \text{~when~} \mu\in\left[ p+1-\sqrt{2p+1},~ p+1 \right).
\end{align*}
For $\mu\in\left[ 1, p+1-\sqrt{2p+1} \right)$, the only obstruction is given by the volume condition.
\end{cor}
\begin{proof} 
On each interval, the reduction algorithm defines an automorphism $\phi\in D_{K}(1,k+1)$ as a composition of Cremona moves and reorderings. The obstructing exceptional class is then given by $\phi^{*} E_{k+1}$. In the present cases, the automorphism is
\[
\phi =
\begin{cases}
(\mathfrak{RC})^{p} 
	& \text{~for~} \mu \in \left[ p+1,~ \infty \right)\\
(\mathfrak{R}\mathfrak{C})(\mathfrak{S}\mathfrak{C})(\mathfrak{R}\mathfrak{C})^{p-2}
	& \text{~for~} \mu \in \left[ p+1-\sqrt{2p+1},~ p+1 \right)
\end{cases}
\]
where $\mathfrak{R}$ and $\mathfrak{C}$ are the automorphisms defined in Remark~\ref{rmk:AutomorphismsTrivialInitialStep}, and where $\mathfrak{S}$ corresponds to the permutation $(1,4,2,3,5,\ldots,k+2)$.
\end{proof}

\subsection{The even case $k=2p$}

This subsection is dedicated to proving the following proposition, which is just a more precise formulation of part $(ii)$ in Theorem~\ref{thm:MainTheoremTrivial}

\begin{prop}\label{thm:MainTheoremTrivialExplicitEven}
 There exist two sequences $a_n $ and $\gamma_n$ satisfying the recurrence relations
\begin{equation}
a_{n+3}=(p-1)a_{n+2}-(p-1)a_{n+1}+a_{n}
\end{equation}
\begin{equation}
\gamma_{n+3}=(p-1)\gamma_{n+2}-(p-1)\gamma_{n+1}+\gamma_{n}
\end{equation}
with initial conditions
\[ 
a_{0} =0, \quad a_{1} =1,\quad a_{2}  = (p-1).
\]
\[
\gamma_{0} = 1, \quad \gamma_{1} = p,\quad \gamma_{2} = (p-1)^2 
\]
so that the generalized Gromov width $\w_{2p}(M_{\mu}^{0})$ is given as a piecewise linear function by       
\[
\w_{2p}(M_{\mu}^{0}) =
\begin{cases}
c_{\vol}=\sqrt{\frac{\mu}{p}}  & \text{~if~}\mu\in\left[ 1,~ \frac{p-2 + \sqrt{p^2 -4p}}{2}\right)\\

\frac{a_{n-1}\mu + a_n}{2 (a_n+a_{n-1}) - 1} & \text{~if~}\mu\in\left[ \frac{\gamma_{n}}{\gamma_{n-1}},~ \frac{\gamma_{n-1}}{\gamma_{n-2}} \right) ,~n\geq 2\\

1 & \text{~if~}\mu\in\left[ p,~ \infty \right)
\end{cases}
\]
\end{prop}
The computations of the obstructing classes, as well as the stability numbers are immediate consequences of the whole argument and thus only appear at the end of this subsection.

To start with, we know from Lemma~\ref{lemma:TrivialBundleEvenOddDichotomy}, that the
vector
\[
  v_{p-3} =
  \left( \lambda_{p-3}\,; \lambda_{p-3}-1,c^{\times 5},(1-c)^{\times (2p-5)} \right)
\]
is ordered whenever $k\geq 4$, $\mu\geq 1$, and $c\in(1/2,
\min\{c_{\vol}, 1\}]$. Thus, we can perform a $\mathfrak{R} \mathfrak{C}$ move to get
\[
  v_{p-2} =
  \left( \lambda_{p-2}\,;\lambda_{p-2}-1, c^{\times 3}, (1-c)^{\times (2p-3)} \right)
\]
If $v_{p-2}$ is ordered, then another $\mathfrak{R} \mathfrak{C}$ move gives
\[
  v_{p-1} =
  \left( \lambda_{p-1}\,;\lambda_{p-1}-1, c^{\times 1}, (1-c)^{\times (2p-1)} \right),
\]
which is not necessarily ordered, but which is non-negative and whose defect is always zero. The vector $v_{p-1}$ is thus reduced, and we conclude that the $(2p)^{\text{th}}$ generalized Gromov width is $\w_{2p}(M_{\mu}^{0}) = \min\{c_{\vol}, 1\}$.

If $v_{p-2}$ is not ordered, then
\begin{equation}\label{vp-2}
  \hat{v}_{p-2} =
  \left(\lambda_{p-2}\,; c^{\times 3}, \lambda_{p-2}-1, (1-c)^{\times (2p-3)}\right)
\end{equation}
is ordered with defect $d' := 3c - \lambda_{p-2} > d_{0} = 2c-1>0$. Performing a Cremona move and reordering the resulting vector gives
\[
  \hat{v}^{(2)}_{1} :=
  \left(\lambda_{p-2}-d'\,; \lambda_{p-2}-1, (1-c)^{\times (2p-3)}, (c-d')^{\times 3}\right),
\]
which is of the same form as $v_{1}$. This new vector is non-negative if and only if $c-d'\geq 0$. Since $d'>d_{0}$, we have $1-c=c-d_{0}> c-d'$, so that $c-d'$ is a new upper bound for $\w_{2p}(M_{\mu}^{0})$. The defect of $\hat{v}_{1}^{(2)}$ is $d_{0}^{(2)}=d'-d_{0}>0$, so the vector is not reduced and we can perform a $\mathfrak{R} \mathfrak{C}$ move that results in a vector $v_{2}^{(2)}$ of the same form as $v_{2}$. We can repeat this process until we reach the vector 
\[
  v_{p-2}^{(2)} =
  \left(\lambda_{p-2}-d'-(p-3)d_{0}^{(2)}\,; \lambda_{p-2}-1-(p-3)d_{0}^{(2)}, (1-c)^{\times 3}, (c-d')^{\times (2p-3)}\right),
\]
which is of the same type as $v_{p-2}$. Again, we have the following alternative: if $v_{p-2}^{(2)}$ is non-negative and ordered, the algorithm gives a reduced and non-negative vector after one more step. Otherwise, the algorithm enters a new cycle that starts with a vector $v_{1}^{(3)}$ of the same form as $v_{1}$. We claim that we get all the possible obstructions to the embedding of $k=2p$ balls in $M_{\mu}^{0}$ by iterating this simple procedure.

For consistency, let us write $v_{p-2}^{(1)}$ for the vector $\hat{v}_{p-2}$ obtained in~\eqref{vp-2}. We now define
\[v_{p-2}^{(i)} =
  \left( A_{i}\,; B_{i}^{3}, C_{i}, D_{i}^{\times (2p-3)}\right),\quad i=1,2
\]
Since $A_{i}=B_{i}+C_{i}+D_{i}$, those vectors are completely determined by the triples $(B_{i},C_{i},D_{i})$. Using this shorthand notation, the automorphism of $D_{K}(1,k+1)$ that takes the vector $v_{p-2}^{(1)}$ to $v_{p-2}^{(2)}$ is represented by the matrix
\[
  T = \begin{pmatrix}
  0      & 0     & 1 \\
  -p + 3 & p - 2 & 0 \\
  -1     & 1     & 1
\end{pmatrix}
\]
of determinant one. For each integer $n\geq 0$, we define
\begin{equation}\label{SequenceTriples}
  V_{n+1} :=
  \left(B_{n+1},C_{n+1},D_{n+1}\right) = T^{n}(B_{1},C_{1},D_{1})
\end{equation}
The sequence $V_{n}$ can be understood by looking at the Jordan normal form of $T$.  When $p=4$, the matrix $T$ has a single eigenvalue $1$ and can be written as $E\cdot\Delta\cdot E^{-1}$ where
\[ 
  \Delta = 
  \begin{pmatrix}
    1 & 1 & 0 \\
    0 & 1 & 1 \\
    0 & 0 & 1
  \end{pmatrix}
  \qquad\text{and}\qquad
  E = 
  \begin{pmatrix}
     1 & 0 & 0 \\
    1 & 1 & 1 \\
    1 & 1 & 0
  \end{pmatrix}.
\]
Hence, the orbit of $(B_{1},C_{1},D_{1})$ is contained in a plane on which $T$ acts as a shear map.

For $p\geq 5$, we can write $T=E\cdot\Delta\cdot E^{-1}$ where
\[ 
  \Delta = 
  \begin{pmatrix}
    1 & 0 & 0 \\
    0 & \lambda & 0 \\
    0 & 0 & \lambdabar
  \end{pmatrix}
  \qquad\text{and}\qquad
  E = 
  \begin{pmatrix}
    1 & \lambdabar & \lambda \\
    1 & p-3 & p-3 \\
    1 & 1 & 1
  \end{pmatrix},
\]
and where 
\[
  \lambda =
  \frac{p-2+\sqrt{p^{2}-4p}}{2}, \qquad \lambdabar = \frac{p-2-\sqrt{p^{2}-4p}}{2}.
\]
Since $\lambda\lambdabar=1$, the orbit of a point $(B,C,D)\in\R^{3}$ under repeated multiplication by $\Delta$ is contained in the standard hyperbola
\[\left\{ yz = CD, ~x= B\right\}.\]
It follows that the orbit of $(B_{1},C_{1},D_{1})$ under repeated multiplication by $T$
traces a hyperbola contained in an affine plane generated by the eigenvectors $(\lambdabar,p-3,1)$ and $(\lambda,p-3,1)$. For all $p\geq 4$, the orbit of an initial triple $(B_{1},C_{1},D_{1})$ may be reduced to a two dimensional system by the change of variables
\[R_{n} = B_n-C_n \quad \text{~and~}\quad S_{n} = C_n-D_n.\]
In particular, $R_{1}$ and $S_{1}$ are then given by
\[R_{1}=1-\mu+(p-1)(2c-1)\quad\text{~and~}\quad S_{1}=\mu-1-(p-2)(2c-1),\]
and we can write
\[
  \begin{pmatrix} R_{n+1} \\ S_{n+1} \end{pmatrix} =
  M \begin{pmatrix} R_n\\ S_n \end{pmatrix}
\]
where $M$ is the matrix
\begin{equation}
 \label{eqn:Matrix}
  M =
  \begin{pmatrix}
    (p-3)  & -1 \\
    -(p-4) & 1
  \end{pmatrix}
\end{equation}
with eigenvalues $\lambda$ and $\overline{\lambda}$.  When $p \geq 5$, the orbit of a general point $(R_1, S_1)$ under repeated multiplication by $M$ lies along a hyperbola whose asymptotes extend in the eigendirections of $\lambda$ and $\lambdabar$. A quick computation gives eigenvectors $\left(1, \lambda-1\right)$ and $\left(1, \lambdabar -1 \right)$. Note that the asymptote $S=(\lambda-1)R$, corresponding to the eigenvalue $\lambdabar$, has
positive slope, while the asymptote $S = (\overline{\lambda}-1)R$ has negative slope.
\begin{lemma}\label{lemma:VolumeAndAsymptotes}
Given $p\geq 4$ and $\mu\geq 1$, let $c\in(1/2,c_{\vol}]$. Then  the initial point $(R_1(c,\mu), S_1(c,\mu))$  sits in a convex region determined by a parabola  tangent to the lines 
$S=(\lambda-1)R$ and $S = (\overline{\lambda}-1)R$. The orbit of the point satisfies $S_n > 0$ for all $n$. Moreover, 
\begin{itemize}
  \item if $R_1\leq 0$ then $R_n\leq 0$ for all $n$,
  \item if $R_1>0$ then there exists $N>0$ such that $R_n\leq 0$ if and only if $n\geq N$.
\end{itemize}
\end{lemma}
\begin{proof}
Assume first $p\geq 5$. Then all the points $(R(c,\mu), S(c,\mu))$ for which $c \leq c_{\vol} = \sqrt{\mu/p}$ sit in the convex region determined by the parametrized parabola $(R(c), S(c)):=\left(R(c,c^2 p), S(c, c^2 p)\right)$, see Figure~\ref{fig1} below. All is left to show is the tangency of this parabola to the lines $S = (\overline{\lambda}-1)R$ and $S=(\lambda-1)R$.

Let us show the tangency to $S=(\lambda-1)R$; the other one is similar. The value $c_{\lambda}$ for which the slope of the parabola is $\lambda-1$ at the point $(R(c_{\lambda}),~ S(c_{\lambda}))$ is
\begin{equation}
  c_{\lambda} = \frac {(p-1)\lambda -1}{ p \lambda} =c_{\vol}(\lambda):=\sqrt{\lambda/p}
\end{equation}
and it is immediate to check that the point $(R(c_{\lambda}), S(c_{\lambda}))$ is on
the line $S(c) -(\lambda-1)R(c)=0$, that~is
\begin{equation}
  S(c_{\lambda}) -(\lambda-1)R(c_{\lambda}) = \frac{\lambda^2 +(p-2) \lambda -1}{p \lambda} = 0.
\end{equation} 
This proves the first assertion. To prove the second statement, we observe that if the point in the upper plane lies above the two asymptotes, so does its hyperbolic orbit. Thus $S_n>0$. On the other hand if the point $(R_1,S_1)$ is in the first quadrant, then its corresponding hyperbola intersects the vertical axis. Since the orbit does not have accumulation points, that finishes the proof in the case~$p\geq 5$. When $p=4$, the matrix~\eqref{eqn:Matrix} becomes
\begin{equation}   
 M =
    \begin{pmatrix}
      1  & -1 \\
      0 & 1
    \end{pmatrix},
\end{equation}
which shears the points in the first quadrant horizontally toward the left. The conclusion follows readily. 
\end{proof}
\begin{figure}[htb!]
\centering%
\includegraphics[scale=0.25]{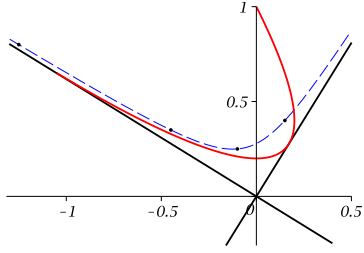}
\caption{The volume curve $(R(c,c^{2}p),~S(c,c^{2}p))$ (solid red curve), eigendirections, and an orbit under iterates of $M$ (dashed blue curve) are graphed for the value $p=5$. In general, the intersections of the volume curve with the vertical axis occur at $c=1$ and $c=(p-2)/p$ which correspond, respectively, to $\mu=p$ and $\mu=(p-2)^{2}/p$. The tangency point with positive slope occurs at $c_{\lambda}$, which corresponds to $\mu=\lambda$. We have packing obstructions for points on the volume curve between $\mu = p$ and $\mu=\lambda$, while we have full packings for points on the volume curve between $\mu=\lambda$ and $\mu = 1\geq (p-2)^{2}/p$.}
\label{fig1}
\end{figure}

For the initial vector $v_{0}=(\mu+1-c\,;\mu-c, c^{\times (2p-1)}, 1-c)$ to belong to the closure of the symplectic cone, it is necessary that all triples $(B_{n},C_{n},D_{n})$ be non-negative. Before we investigate the positivity of these coordinates using the two dimensional picture, we will take a short necessary excursion into the standard theory of recurrent sequences. By the Cayley-Hamilton theorem, the vectors $V_{i}=(B_{i}, C_{i},
D_{i})$ must satisfy the recurrence relation defined by the characteristic polynomial of $T$, namely
\begin{equation}\label{MainRecurrence}
V_{n+3} = (p-1)V_{n+2}-(p-1)V_{n+1}+V_{n}
\end{equation}
It follows that any linear combination of the coefficients $B_{i}$, $C_{i}$, and $D_{i}$ satisfies the same recurrence as well. In particular, we have
\[  D_{n+3} = (p-1)D_{n+2} - (p-1)D_{n+1} + D_{n}\]
with initial conditions
\[
  D_{0} = 0 \quad
  D_{1} = 1-c, \quad
  D_{2} = (p+\mu-1)+(1-2p)c ,\quad D_{3} =
  (p-1)(p+\mu-2)+(-2p^2 + 4p - 1)c.
\]
In fact, since the matrix $T$ is unimodular and has eigenvalues $1,~\lambda,~\overline{\lambda}$, any affine combination of sequences that satisfy the relation will satisfy it too. One more subtle relation is the following lemma: 
\begin{lemma}\label{lemma:crossedsequence}
If two sequences  $x_n, y_n$ satisfy a recurrence of type~\eqref{MainRecurrence} then the sequence 
\begin{equation}\label{thecrossedsequence}
  \phi_n = x_n y_n - x_{n+n_0} y_{n-n_0}
\end{equation}
satisfy the same recurrence.
\end{lemma}
\begin{proof}
We first consider the case $p\geq 5$. We will use the shorthand  $[\lambda]^n$ for the vectors $[\lambda^n,\overline{\lambda}\,^n,1]= [\lambda^n,{\lambda}^{-n},1]$. Since the three eigenvalues $1$, $\lambda$, and $\lambdabar$ are distinct, a sequence $y_n$ satisfies the recurrence~\eqref{MainRecurrence} if, and only if $y_n= [d] \cdot [\lambda]^{n}$ for some row vector $[d]=[d_1, d_2,d_3]$. Then the sequence \[\phi_n =  \left ([d] \cdot [\lambda]^{n} \right )\left ([d'] \cdot [\lambda]^{n}\right )
-  \left ([d] \cdot [\lambda]^{n+n_0}\right )\left ([d'] \cdot [\lambda]^{n-n_0}\right )\] 
 
Note that the terms containing the powers $\lambda^{2n}$ and  $\lambda^{-2n}$ will cancel out; one can easily check that an expansion of the rest of the expression  will be a new linear combination
 \[\phi_n= [d''] \cdot [\lambda]^{n}\] 
and thus verifies the recurrence~\eqref{MainRecurrence}. 

When $p=4$, the characteristic polynomial of the recurrence has a single root of order three. In that case, the general theory of recurrences implies that a sequence satisfies~\eqref{MainRecurrence} if, and only if, it is given by a quadratic polynomial in $n$. The lemma can be easily verified.
\end{proof}
Let us write now $D_{n}=\alpha_{n} - \beta_{n}c$, the coefficients $\alpha_{n}$
and $\beta_{n}$ must satisfy the recurrence~\eqref{MainRecurrence}
with initial conditions
\[
  \alpha_{1}=1, \qquad \alpha_{2}=p+\mu-1, \qquad \alpha_{3}=(p-1)(p+\mu-2),
\]
and
\[
  \beta_{1}=1, \qquad \beta_{2}=2p-1, \qquad \beta_{3}=2p^{2}-4p+1.
\]
 
Thus all components $D_n(c, \mu)$ depend linearly on $c$ so there is a sequence of positive numbers $w_n, n \geq 1$ such that $D_n(c, \mu) >0$ if and only if $c \leq w_n$. The sequence $w_n$ is  obtained as follows:
\begin{equation}\label{thew}
  c \leq w_n = \frac{\alpha_n }{\beta_n},
\end{equation}
with the first initial few given by
\[
  w_{1} = 1,\quad w_{2}=\frac{p+\mu-1}{2p-1},\quad w_{3}= \frac{(p-1)(p+\mu-2)}{2p^{2}-4p+1}.
\]
The sequence $\{w_n\}$ satisfies several identities. For our purpose, one of the most useful is the following alternative definition whose equivalence with~\eqref{thew} can be easily checked by an easy induction argument:
\begin{equation}\label{thewwitha}
  w_n(\mu) = \frac{a_n +a_{n-1} \mu}{2(a_n+a_{n-1})-1}.
\end{equation}
where $\{a_n\}$ is an increasing sequence that satisfies the recurrence~\eqref{MainRecurrence} with initial conditions
\[
  a_0=0,\quad a_{1}=1,\quad a_{2}=p-1,a_{3}=(p-1)(p-2)
\]
The following lemma translates the results obtained by studying the two dimensional linear system $M$ in the variables $(R,S)$ into conclusions about the three dimensional linear system $T$ in the variables $(B,C,D)$.
\begin{lemma}\label{thednbehaviour}
Fix $p \geq 4$, $\mu \geq 1$ and $c$ so that $1/2\leq c\leq c_{\vol}$. Additionally,
restrict to only those pairs $(c,\mu)$ for which  $(B_1,C_1,D_1)$ is in the first octant.
The following facts hold:
\begin{enumerate}
  \item The orbit $(B_n,C_n,D_n)$ of $(B_1,C_1,D_1)$ remains in
    the first octant as long as the coordinate $D_n$ remains positive.
    Thus, the vector $v_{p-2}(c,\mu)$ is in the symplectic cone
    $\mathcal{C}_{K}$ if and only if $D_n(c,\mu) \geq 0$ for all $n
    \in \N$.  

  \item The sequence $D_n$ has an almost monotone behavior, that is,
    only one of the following statements holds:
    \begin{itemize}
      \item The sequences $D_n(c, \mu)$ and $w_n(\mu)$ are strictly
        increasing. Additionally, $c_{\vol}(\mu) \leq w_n(\mu)$ for all $n \in \N$
      \item There exist a natural number $N >1$ such that $D_n(c,
        \mu)$ and $w_n(\mu)$ are decreasing for $n \leq N$ and
        increasing for $n \geq N$. 
    \end{itemize}
\end{enumerate}
\end{lemma}
\begin{proof}
The proofs are immediate. By Lemma~\ref{lemma:VolumeAndAsymptotes} $S_n =C_n -D_n$ is always positive.  It immediately follows that $D_n>0$ implies $C_n>0$. Moreover, since $B_n=D_{n-1}$, it immediately follows that $B_n>0$.

For part (2), recall that the volume condition implies $S_1 > 0$ and $D_{2} = D_1 -(B_1 - C_1)= D_1-R_1$.

In the case that $R_1 > 0$, using Lemma~\ref{lemma:VolumeAndAsymptotes} again, there exists $N > 0$ such that $R_n \leq 0$ if and only if $n \geq N$. Since $D_{n+1} = D_n - R_n$, the sequence $D_n$ is decreasing for $n \leq N$ and increasing for $n \geq N$.  If $R_1 \leq 0$ then as explained in the proof of Lemma~\ref{lemma:VolumeAndAsymptotes} the orbit $(R_n, S_n)$ approaches the asymptote $S = (\overline{\lambda} -1)R$ in the second quadrant, hence $R_n$ remains negative for all $n$. This implies that $D_n$ is always increasing.

Clearly, $w_n(\mu)$ has the same behavior as the sequences $D_n(c,\mu)$. Moreover, in the case when $D_1 <0$, since the sequence $w_n (\mu)$ is increasing, it is sufficient to show that $c_{\vol} \leq w_1 =1$ which is clear. 
\end{proof}
We can now state the main result of this section, namely
\begin{cor}\label{widthismin}
The $(2p)^{\text{th}}$ generalized Gromov width of the trivial bundle $M_{\mu}^{0}$ is 
\[\w_{2p}\left(M_{\mu}^{0}\right)=\min_{i \in \N}\left\{c_{\vol},~w_i(\mu)\right\}\]
where the sequence $w_i$ has at most one minimum. 
\end{cor}
\begin{proof}
Let us first assume that the sequence $\{w_{n}(\mu)\}$ attains a minimum at $n=N>1$. From 
Lemma~\ref{thednbehaviour}, that minimum is positive, and setting $c=w_{N}(\mu)$ in the initial vector
 \[v_{0}=\left(\mu+1-c\,;\mu-c,c^{\times(2p-1)}, 1-c\right)\]
the algorithm produces a sequence of vectors
\[v_{p-2}^{(n)} =
  \left( A_{n}\,; B_{n}^{3}, C_{n}, D_{n}^{\times (2p-3)}\right)
\]
which, by Lemma~\ref{thednbehaviour}, are all nonnegative. For all $n\leq N$ those vectors are ordered since $R_{n}=B_{n}-C_{n}>0$. However, for $n>N$, we have $R_{n}=B_{n}-C_{n}<0$, which shows that, after reordering, $v_{p-2}^{(n)}$ becomes
\[\hat{v}_{p-2}^{(N)} =
  \left( A_{i}\,; C_{i}, B_{i}^{3}, D_{i}^{\times (2p-3)}\right)
\]
Applying a $\mathfrak{R}\mathfrak{C}$ move then yields a reduced, nonegative vector. Consequently, $\w_{2p}(\mu)\geq w_{N}(\mu)$. Since, by construction, each $w_{n}(\mu)$ gives an upper bound on the width $\w_{2p}(\mu)$, we conclude that $\w_{2p}(\mu)=w_{N}(\mu)$.
 
If the sequence $\{w_{i}(\mu)\}$ is increasing, then $c_{\vol} < w_{i}(\mu)$ for all $i\geq 1$. In particular, $\mu$ must belong to the interval $[1,\lambda]$. Setting $c=c_{\vol}(\mu)$ in $v_{0}$, Lemma~\ref{thednbehaviour} shows that the algorithm still produces a sequence of nonnegative vectors
\[v_{p-2}^{(n)} =
  \left( A_{n}\,; B_{n}^{3}, C_{n}, D_{n}^{\times (2p-3)}\right)
\]
As before, those vectors are ordered until $R_{n}=B_{n}-C_{n}<0$. Since $R_{1}\geq 0$ whenever $\mu\geq (p-2)^{2}/p$, and that $1\geq (p-2)^{2}/p$ for $p\geq 4$, Lemma~\ref{thednbehaviour} shows that there exists $N\geq 1$ such that $R_{n}<0$ for all $n>N$. As in the previous case, this implies that the algorithm produces a reduced vector after finitely many steps. Therefore, $\w_{2p}(\mu)=c_{\vol}$.
\end{proof}

In order to write $\w_{2p}(\mu)$ as a piecewise linear function, our next goal is to find an optimal interval $I_n\subset[1,\infty)$ on which $w_n(\mu)$ is the minimum in the sequence $\{w_i(\mu)\}$.

\begin{lemma}\label{muandgamma}
  Let $p\geq 4$ be fixed. Then there exist a sequence $\{\gamma_{n}\}$ given by
\[\gamma_{-1}=0, \qquad \gamma_{0}=1,\qquad \gamma_{1}=p\]

\begin{equation}\label{gammarec}
\gamma_{n+3}=(p-1)\gamma_{n+2}-(p-1)\gamma_{n+1}+\gamma_{n} 
\end{equation}
such that
    \[w_{n+1}\leq w_{n} \iff \mu \leq \frac{\gamma_{n}}{\gamma_{n-1}}.\]
\end{lemma}

\begin{proof}
For ease of writing we will use the notation $\beta_n=2(a_{n-1}+a_n)-1$ which was previously introduced. Notice  that  $w_{n+1}\leq w_{n}$ is equivalent with  $\frac{a_{n+1} +a_{n} \mu}{\beta_{n+1}}   \leq \frac{a_{n} +a_{n-1} \mu}{ \beta_{n} }$. 
This, in turn, translates into

\begin{equation}\label{thegamma}
  \mu \leq
  \frac{ a_n \beta_{n+1} - a_{n+1} \beta_n }
       { a_n \beta_n - a_{n-1} \beta_{n+1} }
\end{equation}
We can prove, by using the Lemma~\ref{lemma:crossedsequence} twice, that both the numerator sequence and denominator sequence satisfy the recurrence~\eqref{MainRecurrence}. We will then define 
\[\gamma_n := a_n \beta_{n+1} - a_{n+1}\beta_n\] 
We leave it to the reader to check that the initial condition are those listed in the statement. To show that the numerator satisfies the recurrence~\eqref{gammarec}, one uses Lemma~\ref{lemma:crossedsequence} with $x_n= a_n, y_n=  \beta_{n+1}$ and $n_0= -1$. For the denominator, the same Lemma~\ref{lemma:crossedsequence} with $x_n = a_n, y_n = \beta_n$ and $n_0=1$ yields that the sequence made with the numerators in~\eqref{thegamma} satisfies the recurrence as well. To show that the denominator is just the numerator sequences with an index shift of 1 it is sufficient to verify this for $n=1,2,3$ using the given initial conditions for $a_n$; we leave this as an exercise. 
\end{proof}

The sequence of quotients $\left( \frac{\gamma_{n}}{\gamma_{n-1}} \right)_{n \in \N}$ is monotone decreasing and converges to $\lambda$. The following computational lemma will be used both for the next results as well as in Section~5.
\begin{lemma}\label{lemma:numerical}
\begin{equation}\label{numerical}
  \frac{(a_n+a_{n-1})^2}{p} - \frac{a_n+a_{n-1}}{p} = a_n a_{n-1}.
\end{equation}
\begin{equation}\label{numerical2}
  \beta_n^2 = 4 p a_n a_{n-1} +1.
\end{equation}
\end{lemma}
\begin{proof}
To see how this holds first note that from the recurrence~\eqref{MainRecurrence} for the sequences $a_n, \beta_n, \gamma_n$ we obtain the general
formulae
\[
  a_n =
  \frac{1}{4-p} \left(-1+
    \frac{p+\sqrt{p^2-4p}}{2p}\lambda^n +
    \frac{p-\sqrt{p^2-4p}}{2p}\overline{\lambda}^n \right)
  \]
\[
  \beta_n =
  \frac{1}{4-p} \left(-p+
    2\lambda^n + 2 \overline{\lambda}^n \right)    
  \]

\[
  \gamma_n =
  \frac{1}{4-p} \left(-2+
    \lambda^{n+1} + \overline{\lambda}^{n+1} \right)    
  \]
Using this equations  as well as the fact that $\lambda \overline{\lambda}=1$ one can verify by a straightforward but lengthy computation the relation~\eqref{numerical2}. The relation~\eqref{numerical} is just an algebraic reformulation of relation~\eqref{numerical2} obtained by completing the square.  Finally, we should point out that in the case $p=4$, the sequences are easy quadratic polynomials, namely $a_n=n^2,\beta_n =(2n-1)^2, \gamma= (n+1)^2$ and the relations above are easily verifiable.
\end{proof}

\begin{lemma}\label{thegammabehaviour}
Let $p\geq 4$ and $n\geq 0$ be a fixed natural numbers. Then
\begin{enumerate}
\item  
\[
w_{n}(\mu) = \min_{i \in \N} \left\{w_i(\mu)\right\} \iff\mu\in\left[\frac{\gamma_{n}}{\gamma_{n-1}},~ \frac{\gamma_{n-1}}{\gamma_{n-2}}\right]
\]
\item 
\[ w_{n}(\mu) < c_{\vol}(\mu) \iff \mu\in\left[\frac{\gamma_{n}}{\gamma_{n-1}},~ \frac{\gamma_{n-1}}{\gamma_{n-2}}\right]
\]
\end{enumerate}
\end{lemma}
\begin{proof} 
The first statement is an easy combination of Lemma~\ref{thednbehaviour} and Lemma~\ref{muandgamma}. For the second part note that the relation is equivalent with  $\frac{a_{n+1} +a_{n} \mu}{\beta_{n+1}}   \leq \sqrt{\frac{\mu}{p}}$ if and only if $\mu\in\left[\frac{\gamma_{n}}{\gamma_{n-1}},~ \frac{\gamma_{n-1}}{\gamma_{n-2}}\right]$. Thus is sufficient to verify that $\frac{\gamma_{n}}{\gamma_{n-1}}$ and $\frac{\gamma_{n-1}}{\gamma_{n-2}}$ are the two roots of the quadratic equation.

\begin{equation}
 \label{thevolq}
  \frac{(a_{n} +a_{n-1} \mu)^2}{\beta^2_{n}}   = \frac{\mu}{p}
\end{equation}
But using the identity~\eqref{numerical2}, the two roots of this equation are of the form $\frac{(\beta_n \pm 1)^2}{4pa^2_{n-1}}$. 
We claim that the general formulae for $a_n,\beta_n,\gamma_n$ can be used to verify the relations 
\begin{equation}
(\beta_{n}+1)^2 \gamma_{n-2}=4p \gamma_{n-1} a_{n-1}^2, (\beta_{n}-1)^2 \gamma_{n-1}=4p \gamma_n a_{n-1}^2, 
\end{equation}
We will omit the computation and simply observe that the two presentations of the roots are equal.
\end{proof}
\begin{remark}\label{rmk:AutomorphismsTrivialFancyStep}
The automorphism $\phi_{n}$ produced by the reduction algorithm when $\mu\in I_n$ is given by
\begin{equation}\label{phi-n}
\phi_{n}:=\left[\mathfrak{BCA}(\mathfrak{RC})^{p-3}\right]^{n-1}\mathfrak{RC}
\end{equation}
where $\mathfrak{R}$ and $\mathfrak{C}$ are defined in Remark~\ref{rmk:AutomorphismsTrivialFancyStep} and where $\mathfrak{A}$ and $\mathfrak{B}$ are the permutation matrices
\[\mathfrak{A}=(1\,; 5,2,3,4,6,\ldots, k+2)\qquad
\mathfrak{B}=(1\,; k, k+1, k+2, 5, \ldots, k-1)\]
\end{remark}\eoe

\begin{remark}\label{rmk:RegionsInTheGraph}
The piecewise linear function $\w_{2p}(\mu)$ approximates $c_{\vol}(\mu)$ from below on the interval $(\lambda, \infty)$. Each function $w_{n}(\mu)$ defines a line in the $RS$ plane, namely $(R(\mu,~w_{n}(\mu)) ,~S(\mu,~w_{n}(\mu))$. That line intersects the volume curve in two points. Any choice of a point in the region delimited by the volume curve and that line yields a vector $v_{0}$ outside the symplectic cone. Each point in the region bounded by the axis $R=0$, the lines $w_{n}$, and the portion of the volume curve between $\mu=1$ and $\mu=\lambda$ gives, after reduction, a nonnegative reduced vector, see Figure~\ref{fig2} below.
\end{remark}\eoe
\begin{figure}[htb!]
\centering%
\includegraphics[scale=0.25]{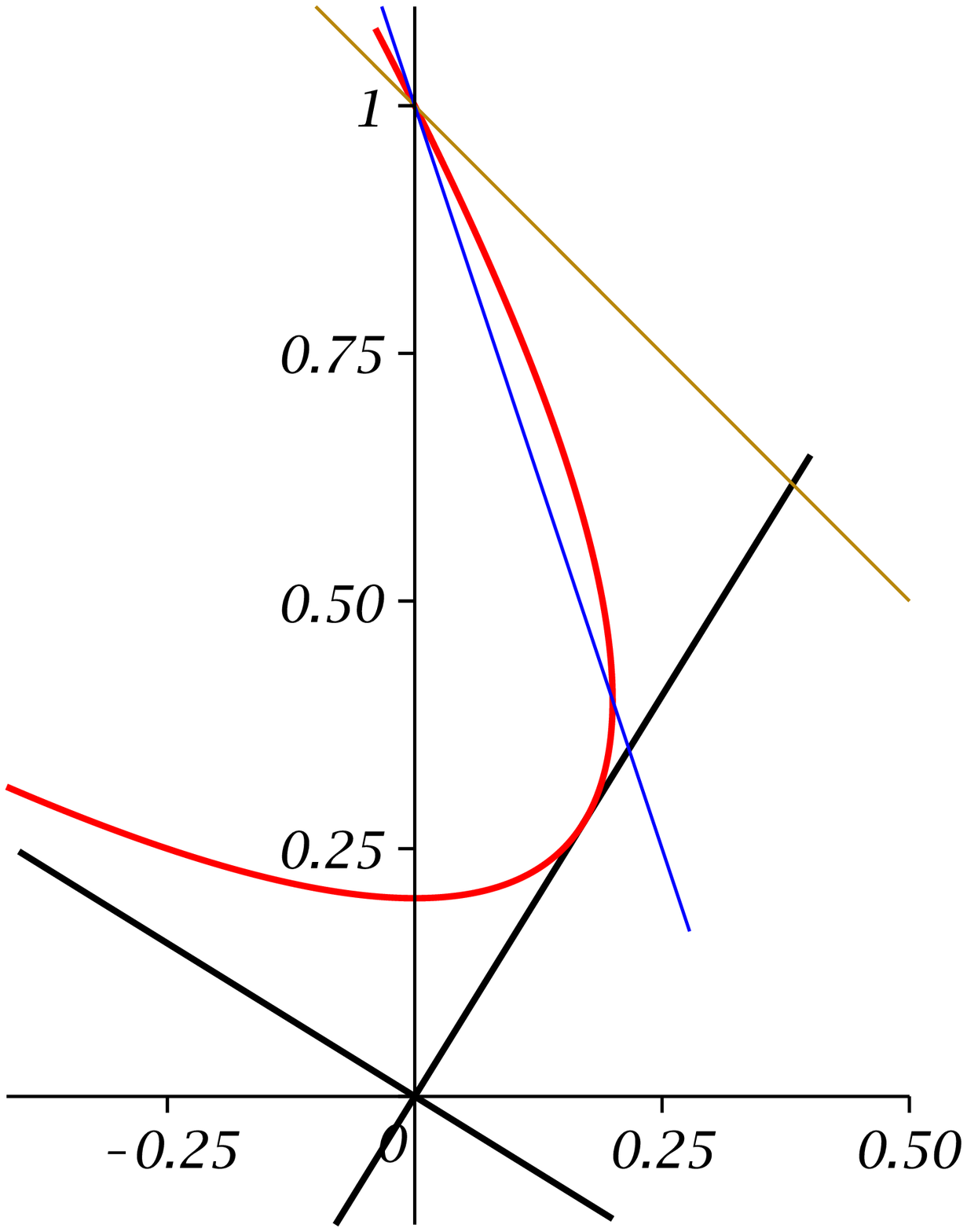}\includegraphics[scale=0.25]{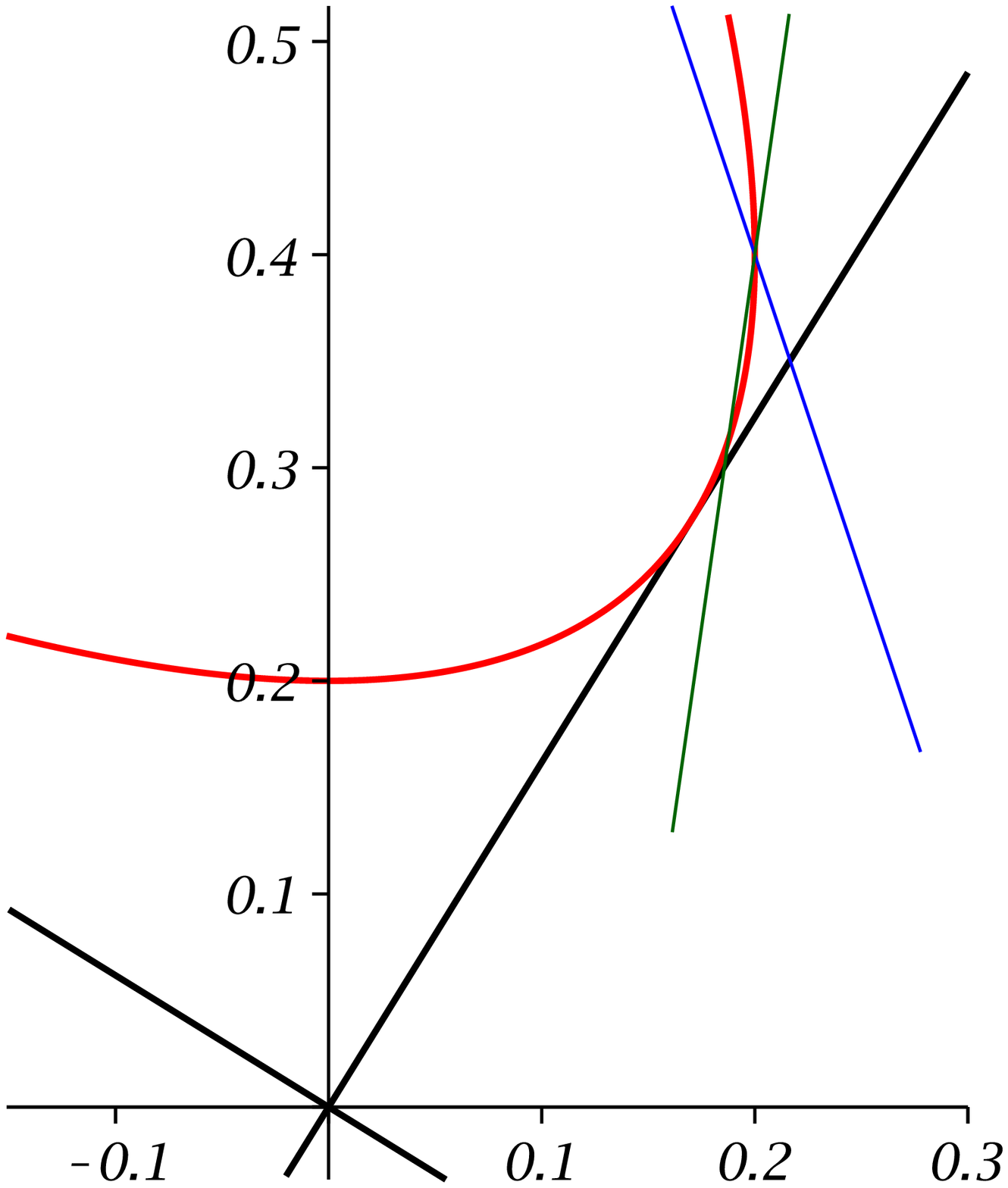}
\caption{The volume curve $(R(c,c^{2}p),~S(c,c^{2}p))$ (solid red curve), together with the lines traced by $w_{1}(\mu)=1$ (orange), $w_{2}(\mu)$ (blue), and $w_{3}(\mu)$ (green) in the $RS$ plane. The piecewise linear function $\w_{2p}(\mu)$ approximates $c_{\vol}(\mu)$ from below on the interval $(\lambda, \infty)$.} 
\label{fig2}
\end{figure}
Observe that Proposition~\ref{thm:MainTheoremTrivialExplicitEven} follows as an immediate consequence of Corollary~\ref{widthismin} and~Corollary~\ref{thegammabehaviour}. Moreover, the computations of the packing numbers and of the stability numbers are easy consequences of~Corollary~\ref{widthismin}.

\begin{cor}\label{cor:MainPackingNumbersTrivialEven}
Let $k=2p\geq 8$ and consider $\mu\geq 1$. Then the  $k^{\text{th}}$
packing numbers of $M_{\mu}^{0}$ are
\[
p_{2p}(M_{\mu}^{0})=
\begin{cases}
1 & \text{~if~}\mu\in\left[ 1,~  \frac{p-2 + \sqrt{p^2 -4p}}{2}    \right)   \\

\frac {p}{\mu} \frac{(a_{n-1}\mu + a_n)^2}{(2 (a_n+a_{n-1}) - 1)^2} & \text{~if~}\mu\in\left[\frac{\gamma_{n}}{\gamma_{n-1}},~ \frac{\gamma_{n-1}}{\gamma_{n-2}}\right)  ,~n\geq 2\\

\frac{p}{\mu} & \text{~if~}\mu\in\left[ p,~ \infty \right)
\end{cases}
\]
\end{cor}
\printlabel{cor:EvenStabilityNumberProduct}
\begin{cor}\label{cor:EvenStabilityNumberProduct}
The even stability number of $M_{\mu}^{0}$ is
$N_{{\rm even}}(M_{\mu}^{0})=2\left\lceil \mu + 2 + \frac{1}{\mu}\right\rceil.$ 
\end{cor}
\begin{proof} In the range $k\geq 8$, the stability number $J(\mu)$ is obtained by solving for $p$ in the equation $c_{\vol}(\mu)=\lambda$, which gives 
\[J(\mu) = \max\left\{8,~2\left\lceil \mu+2+\frac{1}{\mu} \right\rceil\right\}= 2\left\lceil \mu+2+\frac{1}{\mu} \right\rceil \quad \text{whenever~} \mu\geq 1\]
On the other hand, Proposition~\ref{prop:PackingNumbersProduct<8} shows that we also have full packings for the sporadic pairs $(\mu,k)\in\left\{(1,2),~(2,4),~(4/3, 6),~(3,6)\right\}$. However, since $J(4/3)=10$ and $J(3)=12$, we conclude that $N_{{\rm even}}(M_{\mu}^{0})=J(\mu)$.
\end{proof}
Combining the above results with the Remark~\ref{rmk:AutomorphismsTrivialFancyStep} and the strategy presented in Section~\ref{sec:strategy}, we can present the obstruction curves for this case as well:
\begin{cor}\label{cor:TrivialevenCaseObstructionClasses}
Using our identification of the $k$-fold blow-up of $M_{\mu}^{0}$ with $X_{k+1}$, a set of  exceptional classes in $H_{2}(X_{2p+2};\Z)$ giving the
obstructions to the embedding of $2p$ balls into $M_{\mu}^{0}$ is 
\begin{align*}
  \left( 1\,;1^{\times 2}, 0^{\times (2p)}\right) & \text{~when~} \mu\in\left[ p,~ \infty \right);\\
\left(d_{n}\,; z_{n}, y_{n}, x_{n}^{\times (2p-2)}, t_{n}\right)  & \text{~when~} \mu\in \left[\frac{\gamma_{n}}{\gamma_{n-1}},~ \frac{\gamma_{n-1}}{\gamma_{n-2}}\right]
\end{align*}
where the coefficients are given recursively in terms of the sequence $\{a_{n}\}$ by
\[x_n = \frac{2(a_n+a_{n-1}) -1 +(-1)^n}{2p}\]
\begin{gather*}
d_n = a_n + a_{n-1} - x_n~,\quad z_n =a_n - x_n \\
y_n = x_n - (-1)^{n}~,\quad t_n =a_{n-1} - x_n
\end{gather*}
For $\mu\in\left[ 1,   \lambda   \right)$, the only obstruction is given by the volume condition.
\end{cor}
\begin{proof}
On each interval $I_{n}$ the reduction algorithm defines an automorphism $\phi_n \in D_{K}(1, 2p+1)$ such that $\phi_n^{*} E_{k+1}$ is an obstructing exceptional class. From the description of $\phi_{n}$ given in Remark~\ref{rmk:AutomorphismsTrivialFancyStep}, one can see that those classes must be of the form
\[E_{n}:=\left(d_{n}\,; z_{n}, y_{n}, x_{n}^{\times (2p-2)}, t_{n}\right)\]
In order to prove that the formulae for the coefficients given above yield obstructing exceptional classes, we only need to check that (i) $E_{n}\cdot E_{n}=-1$, (ii) $K\cdot E_{n}=1$, and (iii) $v_{0} \cdot E_{n} = D_{n}(\mu,c)$. Indeed, we have
\begin{align*}
E_{n}\cdot E_{n} &= d^{2}_{n}-z^{2}_{n}-y^{2}_{n}-(2p-2)x^{2}_{n}-t^{2}_{n}\\
& = 2a_{n}a_{n-1} -2px^{2}_{n}+2x_{n}(-1)^{n}-1\\
& = 2a_{n}a_{n-1} -\frac{\left(2(a_{n}+a_{n-1})-1\right)^{2}+1}{2p}-1\\
& = 2\,\left(a_{n}a_{n-1}-\frac{(a_{n}+a_{n-1})^{2}}{p}+\frac{(a_{n}+a_{n-1})}{p}\right)-1\\
& = -1
\end{align*}
where the last equality follows from Lemma~\ref{lemma:numerical}. Similarly,
\begin{align*}
K\cdot E_{n} & = 3d_{n}-z_{n}-y_{n}-(2p-2)x_{n}-t_{n}\\
&= 2(a_{n}+a_{n-1}) - 2p\,x_{n}+(-1)^{n}\\
&= 2(a_{n}+a_{n-1}) -\left(2(a_{n}+a_{n-1}) -1 +(-1)^{n}\right)    +(-1)^{n}\\
&=1
\end{align*}
and
\begin{align*}
v_{0}\cdot E_{n}&=(\mu+1-c)d_{n}-(\mu-c)z_{n}-cy_{n}-c(2p-2)x_{n}-(1-c)t_{n}\\
&= \mu(d_{n}-z_{n}) +(z_{n}-y_{n}-(2p-2)x_{n}+t_{n}-d_{n})c + (d_{n}-t_{n})\\
&= a_{n-1}\mu + a_{n} - \left(2\,(a_{n}+a_{n-1})-1\right) c\\
&= D_{n}(\mu,c)
\end{align*}

\end{proof}

\begin{remark}
To illustrate the previous corollary, the obstructing classes correponding to the intervals $I_{2}$, $I_{3}$, and $I_{4}$ are of types
\[\left(p-1\,; p-2, 0, 1^{\times(2p-2)}, 0\right), \quad 
	\left(p^{2}-3p+3\,; (p-2)^{2}, p-1, (p-2)^{\times(2p-2)}, 1\right)\]
\[ \left((p-2)^{2}+(p-2)^{3}\,; (p-1)(p-2)(p-3), (p-1)(p-3), \left((p-2)^{2}\right)^{\times(2p-2)}, p-2\right)\]
\end{remark}

\begin{remark}
Note that the arguments in Corollary~\ref{cor:TrivialevenCaseObstructionClasses} explain why we expressed the functions $w_n(\mu)$ using the formula~\eqref{thewwitha}. Namely, the sequence $\{a_{n}\}$ establishes a direct connection between the bounds $w_n(\mu)$ and Biran's result~\eqref{birdiotriv} presented in terms of the Diophantine equations~\eqref{birdioeq}. As expected, a consequence of finding obstructing classes that give the packing numbers is that we can provide the solutions for the Diophantine minimizing problem described in~\eqref{birdioeq}. When $k$ is odd, the relation between the generalized Gromov widths and the Diophantine equations is particularly easy to see. Indeed, for a fixed $k=2p+1$, and any $n \geq 2$, let us take $n_1= 1$ and $n_2= p$. Then the Diophantine equations~\eqref{birdioeq} have solutions $m_i = 1, i=1,2p+1.$  These solutions correspond exactly to the coefficients of our obstructing curves from Corollary~\ref{cor:TrivialOddCaseObstructionClasses} when translated back to the base of the homology of $S^2 \times S^2$. Thus our results could be interpreted as providing the infimum from the relation~\eqref{birdiotriv} without going through the extremely difficult task of solving all other possible Diophantine equation involved. We should also remark that similar solutions can be provided for all other cases that we discuss in the paper. 
\end{remark}

\section{Embeddings of $k\geq 8$ disjoint balls in the non-trivial bundle $M_{\mu}^{1}$}\label{section:TwistedBundle}

This section is dedicated to providing the proofs of Theorem~\ref{thm:MainTheoremTwisted} and its immediate corollaries. As explained in Section~\ref{sec:strategy}, given $k\geq 8$, and $\mu>0$, our goal is to find the largest capacity $c$ for which the vector $v_{0} = \left(\mu+1\,; \mu, c^{\times k}\right)$ belongs to the closure of the symplectic cone. As before, the volume condition gives an upper bound on $\w_{k}$, namely
\[\w_{k} \leq c_{\vol}=\sqrt{\frac{2\mu+1}{k}} \]
Because $\mu$ can take values in $(0,1)$, we cannot assume $c\leq\mu$, so that $v_{0}$ may not be ordered.

\begin{lemma}
Let $k\geq 8$ and suppose $\w_{k}(\mu)\geq\mu$. Then $\mu\leq 1/2$. Consequently, for $\mu\geq 1/2$, we can assume $c\leq\w_{k}\leq\mu$.
\end{lemma}
\begin{proof}
The inequality $\w_{k}(\mu)\geq\mu$ implies that $c_{\vol}\geq\mu$, which is equivalent to $\mu \in \left( 0,\frac{1+\sqrt{k+1}}{k}\right)$. Now, $\frac{1+\sqrt{k+1}}{k}$ is a decreasing function of $k$ that takes the value $1/2$ at $k=8$.
\end{proof}

Assuming $\mu\geq 1/2$, the vector $v_{0}$ is ordered and positive, with defect $d_{0}=2c-1$, so that $v_{0}$ is reduced whenever $c\leq 1/2$. Consequently, we have the lower bound $\w_{k}\geq 1/2$. We note, in particular, that for $\mu=1/2$ and $k=8$, we have $c_{\vol}=1/2=\mu$, which shows that
\[\w_{8}(1/2) = 1/2\]
For $c>1/2$, applying a sequence of $\mathfrak{R}\mathfrak{C}$ moves leads to vectors $v_{n}$ of the form
\[
  v_{n} =
  \left( \mu+1-nd_{0}\,;\mu-nd_{0},c^{\times (k-2n)},(1-c)^{\times (2n)} \right)
\]
\printlabel{lemma:TwistedVp-2ordered}
\begin{lemma}\label{lemma:TwistedVp-2ordered}
  Given $k\geq 8$, let write $k=2p$ or $k=2p+1$ depending on the parity
  of $k$. Choose any $\mu\geq1/2$ and $c\in(1/2, \min\{1, c_{\vol},
  \mu\}]$. Then the following holds:
  \begin{itemize}
    \item If $k\geq 9$, then the vector $v_{p-2}$ is ordered and positive.
    \item If $k=8$, the vector $v_{2}$ is ordered and positive whenever
      $\mu\geq \frac{7}{4}$.
  \end{itemize}
\end{lemma}
\begin{proof}
The vector $v_{p-2}$ is ordered if and only if $\mu-(p-2)(2c-1)\geq c$. Since we must have $c\leq c_{\vol}=\sqrt{(2\mu+1)/k}$, it is sufficient to assume $c=c_{\vol}$, in which case we get
\[
  f(\mu) :=
  k(\mu+p-2)^{2}-(2\mu+1)(2p-3)^{2} \geq
  0
\]
When $k=2p+1$ is odd, the discriminant of this polynomial is $2(7-2p)$, so that $f$ has no real roots whenever $p\geq 4$ and hence must be positive. When $k=2p$ the discriminant is $2(9-2p)$, showing that $f$ has real roots only for $p=4$, that is, for $k=8$. In that case, the roots are $\{1/2, 7/4\}$.
\end{proof}

Now let assume $\mu\in(0,1/2)$. As before, the vector $v_{0}=(\mu+1\,;\mu,c^{\times k})$ is ordered only if $\mu \geq c$, in which case its defect is $2c-1$. Hence, $v_{0}$ is positive and reduced whenever $0<c<\mu\leq 1/2$. On the other hand, when $c>\mu$, the reordering of $v_{0}$ gives the vector $\hat{v}_{0}=(\mu+1\,; c^{\times k},\mu)$ with defect $3c-\mu-1$. Hence, that vector is positive and reduced whenever $0<\mu<c\leq (\mu+1)/3$. Now, we have
$c_{\vol}\leq (\mu+1)/3$ if and only if $k\mu^{2}+2(k-9)\mu+(k-9)\geq 0$, which is true whenever $k\geq 9$. Therefore,
\printlabel{lemma:TwistedWkgeq9}
\begin{lemma}\label{lemma:TwistedWkgeq9}
Assume $k\geq 9$ and $\mu\in (0,1/2]$. Then $\w_{k}(M_{\mu}^{1})=c_{\vol}$, that is, we have full packing of $M_{\mu}^{1}$ by $k$ equal balls. For $k=8$ and $\mu\in (0,1/2]$, we have the lower bound $\mu\leq \w_{8}(M_{\mu}^{1})$ with equality when $\mu=1/2$.
\end{lemma}
We now discuss the following cases separately:
\begin{itemize}
  \item $k=2p+1\geq 9$ and $\mu > 1/2$
  \item $k=2p\geq 10$ and $\mu > 1/2$
  \item $k=8$ 
\end{itemize}

\subsection{The odd case $k=2p+1\geq 9$ and $\mu > 1/2$} 
By Lemma~\ref{lemma:TwistedVp-2ordered}, the vector 
\[
  v_{p-2} =
  \left( \mu+1-(p-2)d_{0}\,;\mu-(p-2)d_{0},c^{\times 5},(1-c)^{\times (2p-4)} \right)
\]
is ordered and positive. A $\mathfrak{R}\mathfrak{C}$ move leads to
\[
  v_{p-1} =
  \left( \mu+1-(p-1)d_{0}\,;\mu-(p-1)d_{0},c^{\times 3},(1-c)^{\times (2p-2)} \right)
\]
which is positive but not necessarily ordered. 

If $v_{p-1}$ is ordered, that is, if $\mu-(p-1)d_{0}\geq c$, then its defect is $d_{0}=2c-1>0$ and another $\mathfrak{R}\mathfrak{C}$ move yields
\[
  v_{p} =
  \left(\mu+1-pd_{0}\,;\mu-pd_{0},c,(1-c)^{\times (2p)} \right)
\]
which, again, is positive but not necessarily ordered. However, since $\mu-(p-1)d_{0}\geq c$ and $(c-1)=c-d_{0}$ implies $\mu-pd_{0}\geq (1-c)$, its defect is
\[(\mu-pd_{0})+c+(1-c)-(\mu+1-pd_{0}) = 0\]
so that $v_{p}$ is positive and reduced.

If $v_{p-1}$ is not ordered, that is, if $\mu-(p-1)d_{0}< c$, then the reordered vector is
\[
  \hat{v}_{p-1} =
  \left(\mu+1-(p-1)d_{0}\,;c^{\times 3},\mu-(p-1)d_{0},(1-c)^{\times (2p-2)} \right)
\]
with defect $d_{p-1}=(2p+1)c-\mu-p>d_{0}$. Hence, applying a Cremona move and a reordering gives
\[
  v_{p} =
  \left( 2\lambda_{p-1}-3c\,; \lambda_{p-1}-1, (1-c)^{\times
    (2p-2)},(\lambda_{p-1}-2c)^{\times 3} \right)
\]
where we have set $\lambda_{p-1}=\mu+1-(p-1)d_{0}$. The vector $v_{p}$ is always ordered, and it is non-negative if and only if $\lambda_{p-1}-2c\geq 0$, which is equivalent to
\[c \leq \frac{p+\mu}{2p}\]
The defect of $v_{p}$ is $d_{p}=1+c-\lambda_{p-1}$, which is positive if and only if we assume $v_{p-1}$ not ordered. Thus, $v_{p}$ is not reduced, and we can apply a sequence of $(p-3)$ $\mathfrak{R}\mathfrak{C}$ moves to obtain the vector
\[
  v_{2p-3} =
  \left( 2\lambda_{p-1}-3c-(p-3)d_{p}\,;
  \lambda_{p-1}-1-(p-3)d_{p},(1-c)^{\times
    4},(\lambda_{p-1}-2c)^{\times (2p-3)} \right)
\]
That vector is ordered if $\lambda_{p-1}-1-(p-3)d_{p}\geq(1-c)$, which is equivalent to
\[c \leq \frac{(p-2)\mu+p^{2}-3p+1}{p(2p-5)}\]
For $c = c_{\vol}$, this becomes
\[
  \left( 2p^{3} - 7p^{2} + 4p + 4 \right) u^{2} + 36p^{3} - 2 \left( 2p^{4} - 11p^{3} + 16p^{2} - 3p + 2\right) u + 2p^{5} - 15p^{4} - 26p^{2} - 4p + 1 \geq
  0
\]
which has a double root for $p=4$ and no real roots for $p\geq 5$. Hence, $v_{2p-3}$ is ordered for all $c\leq c_{\vol}$. A $\mathfrak{R}\mathfrak{C}$ move applied to $v_{2p-3}$ then yields
\[
  v_{2p-2} =
  \left( 2\lambda_{p-1}-3c-(p-2)d_{p}\,;
  \lambda_{p-1}-1-(p-2)d_{p},(1-c)^{\times
    2},(\lambda_{p-1}-2c)^{\times (2p-1)} \right)
\]
which is positive but not necessarily ordered. 

Assuming $v_{2p-2}$ ordered, a last $\mathfrak{R}\mathfrak{C}$ move gives
\[
  v_{2p-1} =
  \left( 2\lambda_{p-1}-3c-(p-1)d_{p}\,; \lambda_{p-1}-1-(p-1)d_{p},
  (\lambda_{p-1}-2c)^{\times (2p+1)} \right)
\]
which is positive with defect $d_{2p-1}=\lambda_{p-1}-1-c=-d_{p}$. Hence, $v_{2p-1}$ is reduced.

If $v_{2p-2}$ is not ordered, then the reordered vector is
\[
  \hat{v}_{2p-2} =
  \left( 2\lambda_{p-1}-3c-(p-2)d_{p}\,; (1-c)^{\times
    2},\lambda_{p-1}-1-(p-2)d_{p},(\lambda_{p-1}-2c)^{\times (2p-1)}
  \right)
\]
and another Cremona move and reordering gives
\[
  v_{2p-1} =
  \left( 2\lambda_{p-1}-3c-(p-1)d_{p}\,; (\lambda_{p-1}-2c)^{\times
    (2p+1)} , \lambda_{p-1}-1-(p-1)d_{p}\right)
\]
which is non-negative only if $\lambda_{p-1}-1-(p-1)d_{p}\geq 0$, which
is equivalent to
\[c \leq \frac{p(p+\mu-1)}{2p^{2}-p-1}\] 
Its defect is $d_{2p-1}=3(\lambda_{p-1}-2c)-(2\lambda_{p-1}-3c-(p-1)d_{p})$ which is positive if $c\geq \frac{(p-2)u + p^{2} - 3p + 1}{2p^{2} - 5p}$, that is, if $v_{2p-3}$ is not ordered. So, $v_{2p-1}$ is again reduced.

Summing up, we see that the algorithm produces a non-negative reduced vector provided $c$ is not greater than any of the upper bounds $c_{\vol}$, $1$, $\frac{p+\mu}{2p}$, and
$\frac{p(p+\mu-1)}{2p^{2}-p-1}$.
\printlabel{prop:TwistedOddGeneric}
\begin{prop}\label{prop:TwistedOddGeneric}
Let $k=2p+1\geq 9$ and $\mu\in(1/2,\infty)$. Then the
$(2p+1)^{\text{th}}$ generalized Gromov width of $M_{\mu}^{1}$ is
\[
  \w_{2p+1}(M_{\mu}^{1}) =
  \begin{cases}
    c_{\vol} & \text{~if~} \mu\in \left[1/2,~ \frac{p^{3}-2p^{2}+1-(p-1)\sqrt{2p+1}}{p^{2}}\right)\\
    \frac{p(p+\mu-1)}{2p^{2}-p-1} & \text{~if~} \mu\in \left[\frac{p^{3}-2p^{2}+1-(p-1)\sqrt{2p+1}}{p^{2}},~ \frac{p(p-1)}{p+1}\right)\\
    \frac{p+\mu}{2p} & \text{~if~} \mu\in\left[\frac{p(p-1)}{p+1} ,~ p\right) \\
    1 & \text{~if~} \mu\in\left[p,~\infty\right)\\
  \end{cases}
\]
\end{prop}
\begin{proof}
  This follows readily from the fact that $\w_{2p+1}=\min\left\{c_{\vol}, 1, \frac{p+\mu}{2p}, \frac{p(p+\mu-1)}{2p^{2}-p-1}\right\}$.
\end{proof}
\begin{cor}
Let $k=2p+1\geq 9$ and consider $\mu\geq 1/2$. The $(2p+1)^{\text{th}}$ packing number of $M_{\mu}^{1}$ is
\[
  p_{2p+1}(M_{\mu}^{1}) =
  \begin{cases}
    1  &\text{~if~}\mu\in\left[1/2,~ \frac{p^{3}-2p^{2}+1-(p-1)\sqrt{2p+1}}{p^{2}}\right) \\
     \frac{p^{2}(p+\mu-1)^{2}}{(2\mu+1)(p-1)^{2}} & \text{~if~} \mu\in \left[\frac{p^{3}-2p^{2}+1-(p-1)\sqrt{2p+1}}{p^{2}},~ \frac{p(p-1)}{p+1}\right)\\
    \frac{(2p+1)}{(2u+1)}\left(\frac{p+\mu}{2p}\right)^{2} & \text{~if~} \mu\in\left[\frac{p(p-1)}{p+1} ,~ p\right) \\
     \frac{2p+1}{2\mu+1} & \text{~if~} \mu\in\left[p,~\infty\right)\\
  \end{cases}
\]
\end{cor}
\cqfd
\printlabel{cor:OddStabilityNumbersTwisted}
\begin{cor}\label{cor:OddStabilityNumbersTwisted}
The odd stability number of $M_{\mu}^{1}$ is
\[
N_{{\rm odd}}(\mu)=
\begin{cases}
7 & \text{~if~} \mu\in\left\{\frac{1}{7},~ \frac{3}{8} \right\} \\
9 & \text{~if~} \mu\in
		\left(0,~1\right)\setminus\left\{\frac{1}{7},~ \frac{3}{8} \right\} \\
2\left\lceil \frac{u + 2 + \sqrt{(u+2)^{2} + 4\sqrt{2u + 1}}}{2} \right\rceil+1
		& \text{~if~} \mu \in \left[1,\infty\right)
\end{cases}
\]
\end{cor}
\begin{proof}
By Proposition~\ref{prop:PackingNumbersTwisted<8}, the pairs $(\mu,k)$ for which we have full packings by $k=2p+1\leq 7$ balls are $\left\{(1, 3),~ (1/7, 7),~ (3/8, 7),~ (3, 7)\right\}$. On the other hand, for $\mu\in(0,1/2]$, Lemma~\ref{lemma:TwistedWkgeq9} shows that we have full packings whenever $k\geq 9$. These two facts together prove our claim for $\mu\in(0,1/2]$. When $\mu\in(1/2, \infty)$, the largest root of the polynomial in $p$
\[
\left(2p+1\right)\left(p^{4}-2p^{3}\mu+p^{2}\mu^{2}-4p^{3}+4p^{2}\mu+4p^{2}-2\mu-1\right)
\]
obtained by setting
\[c_{\vol}^{2}=\frac{p^{2}(p+\mu-1)^{2}}{(2p^{2}-p-1)^{2}}\]
is
\[r(\mu) =  \frac{\mu + 2 + \sqrt{(\mu+2)^{2} + 4\sqrt{2\mu + 1}}}{2}\]
The integer $J(\mu) := \max \{9,~ 2\lceil r(\mu) \rceil + 1\}$ gives the odd stability number in the range $k\geq 9$. The results follows by comparing $J(\mu)$, $\mu\geq 1/2$, with the exceptional full packings by $k\leq 7$ balls listed above.
\end{proof}
\printlabel{cor:TwistedOddCaseObstructionClasses}
\begin{cor}\label{cor:TwistedGenericOddCaseObstructionClasses}
The exceptional classes in $H_{2}(X_{2p+2};\Z)$ that give the obstructions to the embedding of $2p+1\geq 9$ balls into $M_{\mu}^{1}$, $\mu\geq 1/2$, are of type
\begin{align*}
\left(1\,;1^{\times 2},0^{\times 2p}\right) 
  	& \text{~for~} \mu\in[p,~\infty)\\
\left(p\,; p-1, 1^{\times 2p}, 0\right) 
  	& \text{~for~} \mu\in\left[\frac{p(p-1)}{p+1},~ p\right)\\
\left(p(p-1)\,; p(p-2), (p-1)^{\times (2p+1)}\right) 
  	& \text{~for~} \mu\in\left[\frac{p^{3}-2p^{2}+1-(p-1)\sqrt{2p+1}}{p^{2}},~ \frac{p(p-1)}{p+1}\right)
\end{align*}
For $\mu\in\left[1/2,~ \frac{p^{3}-2p^{2}+1-(p-1)\sqrt{2p+1}}{p^{2}}\right)$, the only obstruction is given by the volume condition.
\end{cor}
\begin{proof} 
As in Corollary~\ref{cor:TrivialOddCaseObstructionClasses}, this follows from applying to the vector $(0\,;0^{\times (2p+1)}, -1)$ the adjoint of the automorphism $\phi$ produced by the algorithm on each interval. That automorphism is
\[
\phi = 
\begin{cases}
(\mathfrak{RC})^{p} 
	&  \text{~for~} \mu\in[p,~\infty)\\
(\mathfrak{RC})^{p-1}\mathfrak{BCA}(\mathfrak{RC})^{p-1} 
	& \text{~for~} \mu\in\left[\frac{p(p-1)}{p+1},~ p\right)\\
\mathfrak{BCS}(\mathfrak{RC})^{p-2}\mathfrak{BCA}(\mathfrak{RC})^{p-1} 
	& \text{~for~} \mu\in\left[\frac{p^{3}-2p^{2}+1-(p-1)\sqrt{2p+1}}{p^{2}},~ \frac{p	
		(p-1)}{p+1}\right)
\end{cases}
\]
\end{proof}
%

\subsection{The even case $k=2p\geq 10$ and $\mu > 1/2$}
Let $d_{0}=2c-1$. By Lemma~\ref{lemma:TwistedVp-2ordered}, the vector
\[
  v_{p-2} =
  \left(\mu+1-(p-2)d_{0}\,;\mu-(p-2)d_{0},c^{\times 4},(1-c)^{\times (2p-4)} \right)
\]
is ordered and positive. A $\mathfrak{R}\mathfrak{C}$ move leads to
\[
  v_{p-1} =
  \left(\mu+1-(p-1)d_{0}\,;\mu-(p-1)d_{0},c^{\times 2},(1-c)^{\times (2p-2)} \right)
\]
which is positive but not necessarily ordered. In any case, its defect is still $d_{0}$ so that a $\mathfrak{R}\mathfrak{C}$ move yields
\[
  v_{p} =
  \left(\mu+1-pd_{0}\,;\mu-pd_{0},(1-c)^{\times 2p} \right)
\]
which is non-negative only if $\mu-pd_{0}\geq 0$, which is equivalent to
\[c \leq \frac{p+\mu}{2p}\]
If $v_{p}$ is non-negative and ordered, then its defect is zero, so that $v_{p}$ is reduced. If $v_{p}$ is non-negative but not ordered, then its reordering gives
\[
  \hat{v}_{p} =
  \left(\mu+1-pd_{0}\,;(1-c)^{\times 2p}, \mu-pd_{0} \right)
\]
whose defect is $d_{p}=3-3c-(\mu+1-pd_{0})$. That defect is positive only if
\[c > \frac{p+u-2}{2p-3}\]
However, that would imply
\[c_{\vol} = \sqrt{\frac{2\mu+1}{2p}} > \frac{p+u-2}{2p-3}\]
which is impossible for $k=2p\geq 10$. Hence, $\hat{v}_{p}$ is also reduced.

The previous discussion shows that the algorithm produces a non-negative reduced vector provided $c$ is not greater than any of the upper bounds $c_{\vol}$, $1$, and $\frac{p+\mu}{2p}$.
\printlabel{prop:TwistedEvenGeneric}
\begin{prop}\label{prop:TwistedEvenGeneric}
Let $k=2p\geq 10$ and $\mu\in(1/2,\infty)$. Then the $(2p)^{\text{th}}$ generalized Gromov width of $M_{\mu}^{1}$ is
\[
  \w_{2p}(M_{\mu}^{1}) =
  \begin{cases}
    c_{\vol} & \text{~if~} \mu\in \left[1/2,~ p-\sqrt{2p}\right)\\
    \frac{p+\mu}{2p} & \text{~if~} \mu\in\left[p-\sqrt{2p} ,~ p\right) \\
    1 & \text{~if~} \mu\in\left[p,~\infty\right)\\
  \end{cases}
\]
\end{prop}
\begin{proof}
This follows readily from the fact that $\w_{2p}=\min\left\{ c_{\vol}, 1, \frac{p+\mu}{2p} \right\}$.
\end{proof}
\begin{cor}
Let $k = 2p\geq 10$ and consider $\mu\geq 1/2$. The $(2p)^{\text{th}}$ packing number of $M_{\mu}^{1}$ is
\[
  p_{2p}(M_{\mu}^{1}) =
  \begin{cases}
    1  & \text{~if~}\mu\in\left[1/2,~ p-\sqrt{2p}\right) \\
    \frac{(p+\mu)^{2}}{(2p)(2\mu+1)} & \text{~if~} \mu\in\left[p-\sqrt{2p} ,~ p\right) \\
    \frac{2p}{2\mu+1} & \text{~if~} \mu\in\left[p,~\infty\right)\\
  \end{cases}
\]
In particular, the even stability number of $M_{\mu}^{1}$ is $N_{{\rm even}}(\mu)=2p$ where $p=\left\lceil \mu+1+\sqrt{2\mu+1}
\right\rceil$.
\end{cor}
\begin{proof}
The stability number is obtained by solving for $p$ in the polynomial $c_{\vol}^{2}-\frac{(p+\mu)^{2}}{(2p)^{2}}$. Since this is a degree two polynomial with negative leading term, choosing the largest root gives the result.
\end{proof}
\printlabel{cor:TwistedGenericEvenCaseObstructionClasses}
\begin{cor}\label{cor:TwistedGenericEvenCaseObstructionClasses}
The exceptional classes that give the obstructions to the embedding of $2p\geq 10$ balls into $M_{\mu}^{1}$, $\mu\geq 1/2$, are of type
\begin{align*}
  \left(1\,;1^{\times 2},0^{\times (2p-1)}\right) & \text{~for~} \mu\in[p,~\infty)\\
  \left(p\,; p-1, 1^{\times 2p}\right) & \text{~for~} \mu\in\left[p-\sqrt{2p},~ p\right)
\end{align*}
For $\mu\in\left[1/2,~ p-\sqrt{2p}\right)$, the only obstruction is given by the volume condition.
\end{cor}
\begin{proof} 
The obstructing classes are $\phi^{*}(E_{2p+1})$ where 
\[\phi = 
\begin{cases}
(\mathfrak{RC})^{p} & \text{~for~} \mu\in[p,~\infty)\\
\mathfrak{D}(\mathfrak{RC})^{p} & \text{~for~} \mu\in \left[p-\sqrt{2p},~ p\right)
\end{cases}
\]
where $\mathfrak{D}$ is the permutation $(1,k+2,2,\ldots,k+1)$. 
\end{proof}
%

\subsection{The case $k=8$}

When $k=8$,  Lemma~\ref{lemma:TwistedVp-2ordered} and Lemma~\ref{lemma:TwistedWkgeq9} show that the behaviour of the algorithm depends on whether $\mu\leq 1/2$ or $1/2\leq \mu$. In the first case, the lower bound $\mu\leq\w_{8}(\mu)$ implies that one must start with the ordered vector $\left( \mu+1\,; c^{\times 8}, \mu\right)$, while in the second case, the upper bound $\w_{8}(\mu)\leq \mu$ show that the initial vector is  $\left( \mu+1\,;\mu,c^{\times 8}\right)$. Moreover, for $\mu<\frac{7}{4}$, the initial steps are sensitive to the actual value of $\mu$. In fact, for $\mu\in(0,\frac{7}{4})$, the branching pattern of the algorithm becomes surprisingly hard to analyze. Consequently, we use a different approach which gives directly the exceptional classes defining the obstructions to the embedding of $8$ balls of capacity $c$ in $M_{\mu}^{1}$. This approach relies on the classical fact that the set of exceptional classes of $\CP^{2}\nblowup{9}$ can be described in terms of the affine root lattice of type $E_{8}$ and, as such, it only applies to the case $k=8$. 

To begin with, we show that the exceptional classes leading to embedding obstructions must be ``almost parallel'' to the vector $w=\left(\mu+1\,; \mu, (c_{\vol})^{\times 8}\right)$, see also~\S2 in~\cite{MS}.
\printlabel{lemma:ExplicitTypes}
\begin{lemma}\label{lemma:ExplicitTypes} The classes that may give obstructions to the embeddings of $8$ equal balls in $M_{\mu}^{1}$ are of the forms
\renewcommand{\labelenumi}{\roman{enumi})}
\renewcommand{\theenumi}{\roman{enumi})}
\begin{enumerate}
\item $(d\,; m, \ell-1, \ell^{\times 7})$
\item $(d\,; m, \ell^{\times 8})$
\item $(d\,; m, \ell+1, \ell^{\times 7})$
\end{enumerate}
\end{lemma}
\begin{proof}
Fix $\mu$ and $c\in(0,c_{\vol}]$ and define $v_{0}=\left(\mu+1\,;\mu, c^{\times 8}\right)$ and $w=\left(\mu+1\,; \mu, (c_{\vol})^{\times 8}\right)$. By definition of $c_{\vol}$ we have $w\cdot w=0$. Suppose an exceptional class $E=(e_{0}\,;e_{1},\ldots,e_{8})$ defines an obstruction, that is, suppose $v_{0}\cdot E \leq 0$. Then we must have $w\cdot E\leq 0$ as well.
If we define $\epsilon=(0, \epsilon_{1},\ldots,\epsilon_{8})$ by setting
\[E = \frac{e_{0}}{\mu+1} w + \epsilon\]
we can write
\[w\cdot E = \frac{e_{0}}{\mu+1} w\cdot w+\frac{e_{0}}{\mu+1}w\cdot\epsilon = \frac{e_{0}}{\mu+1} w\cdot\epsilon \leq 0\]
and
\[-1 = E\cdot E = \left(\frac{e_{0}}{\mu+1}\right)^{2} w\cdot w + 2\left(\frac{e_{0}}{\mu+1}\right) w\cdot \epsilon + \epsilon\cdot \epsilon \]
Those two equations imply that $\parallel\epsilon\parallel^{2} = -\epsilon\cdot\epsilon\leq 1$. Now, we observe that $\parallel\epsilon\parallel$ is simply the Euclidean distance between the vectors $\frac{e_{0}}{\mu+1}w$ and $E$. In particular, the distance between the truncated vectors $\frac{e_{0}}{\mu+1}(c_{\vol},\ldots, c_{\vol})$ and $(e_{2},\ldots e_{8})$ is bounded above by one, that is,
\[\sum_{i=2}^{8} \left(e_{i}-\frac{e_{0}c_{\vol}}{\mu+1}\right)^{2}\leq 1\]
Since the coefficients of $E$ are integers, that implies  the coefficients $\{e_{2},\ldots, e_{8}\}$ must all be equal to some integer $\ell$, with at most one exception, in which case the other coefficient must be $\ell\pm 1$. Since the product $v_{0}\cdot E$ is constant  under permuting the coefficients $\{e_{2},\ldots, e_{8}\}$, we can assume that $E$ is of the form $(d\,; m, \ell-1, \ell^{\times 7})$, $(d\,; m, \ell^{\times 8})$, or $(d\,; m, \ell+1, \ell^{\times 7})$ for some positive integers $d$ and~$m$.
\end{proof}
In order to list the exceptional classes of types (i), (ii), and (iii) above, it is useful to describe the set $\E_{9}$ of all exceptional classes in $X_{9}=\CP^{2}\nblowup{9}$ in a more concrete way. To this end, recall that the $(-2)$-homology classes $\alpha_{i}$ are defined of the standard basis $\{L,E_{1},\ldots,E_{9}\}$ by
\begin{align*}
\alpha_{0} & :=L-E_{1}-E_{2}-E_{3}\\
\alpha_{i} & :=E_{i}-E_{i+1},\quad 1\leq i\leq 7
\end{align*}
and that the Poincaré dual of the first Chern class is given by
\[K := 3L-E_{1}-\cdots-E_{9}\]
We now define the root lattice $Q_{8}\subset H_{2}(X_{9};\Z)$ by setting
\[Q_{8}:=\oplus_{i=0}^{7}\Z\alpha_{i} \simeq \Z^{8}\]
It is known (see, for instance,~\cite{HL}) that there exists a natural bijection $T: Q_{8}\to\E_{9}$ between the root lattice and the set of exceptional classes, namely
\[
T(\alpha) =  E_{9}-\alpha-\frac{1}{2}(\alpha\cdot\alpha)K
\]
whose inverse is given by
\[
T^{-1}(E) = E_{9} - E + (1+E\cdot E_{9})K
\]
Under that bijection, the curves of types (i), (ii), and (iii) take a very simple form, and that allows us to write them explicitely.
\printlabel{lemma:ExplicitClasses}
\begin{lemma}\label{lemma:ExplicitClasses} The classes that may give obstructions to packings by $8$ balls belongs to three families that can be parametrized as follows:
\renewcommand{\labelenumi}{\roman{enumi})}
\renewcommand{\theenumi}{\roman{enumi})}
\begin{enumerate}
\item $\left(n(12n-1)\,; n(4n-3), 4n^{2}-1, (4n^{2})^{\times 7}\right)$
\item $\left(4n(3n+2)\,; 4n^{2}-1, (n(4n+3)^{\times 8}\right)$
\item $\left((3n+2)(4n+3)\,; n(4n+3), 2(n+1)(2n+1)+1, (2(n+1)(2n+1))^{\times 7}\right)$
\end{enumerate}
\end{lemma}
\begin{proof}
We first consider classes of type $(d; m, \ell-1, \ell^{\times 7})$. The bijection $\E_{9}\to Q_{8}\simeq \Z^{8}$ maps any such class to a vector of the form
\[\left( -d+3\ell+3, 2d-6\ell+2, d-3\ell+3, 5,4,3,2,1\right) \in\Z^{8}\]
Writing $n=3\ell - d$, that vector becomes
\[\left( n+3, 2-2n, 3-n, 5,4,3,2,1\right) \in\Z^{8}\]
showing that classes of type $(d\,; m, \ell-1, \ell^{\times 7})$ form a 1-parameter family indexed by $n\in\Z$. Applying the inverse bijection, we obtain an explicit parametrization of elements of $\E_{9}$, namely
\[\left(n(12n-1)\,; n(4n-3), 4n^{2}-1, (4n^{2})^{\times 7}\right),\quad n\in\Z \]
Similarly, one can check that classes of the types $(d\,; m, \ell^{\times 8})$ and $(d\,; m, \ell+1, \ell^{\times 7})$ correspond to vectors
\[\left( n+3, 1-2n, 3-n, 5,4,3,2,1\right)\quad\text{and}\quad\left( n+3, -2n, 3-n, 5,4,3,2,1\right)\]
in~$\Z^{8}$. The formulae in~(ii) and~(iii) follow readily.
\end{proof}

The previous two lemmas show that a necessary and sufficient condition for $v_{0}=\left(\mu+1\,; \mu, c^{\times 8}\right)$, with $0<c<c_{\vol}$, to belong to the symplectic cone of $X_{9}$ is the positivity of the symplectic areas of the classes of types (i), (ii), and (iii). Note that for $n=0$, we obtain curves of types $\left(0\,;0,-1,0^{\times 7}\right)$, $\left(0\,;-1,0^{\times 8}\right)$, and $\left(6\,;0,3,2^{\times 7}\right)$. The first two give the trivial lower bound $\w_{8}(\mu)>0$, while the third gives the upper bound $w_{8}(\mu)\leq\frac{6\mu+6}{17}$. For $n\in\Z\setminus\{0\}$, we get three families of upper bounds for $w_{8}$, namely
\[
u_{1}(\mu,n)=\frac{2n(4n+1)\mu +12n^{2}-n}{32n^{2}-1}, \quad
u_{2}(\mu,n)=\frac{(8n^{2}+8n+1)\mu +12n^{2}+8n}{8n(4n+3)}
\]
and
\[
u_{3}(\mu,n)=\frac{2(4n^{2}+7n+3)\mu +12n^{2}+17n+6}{32n^{2}+48n+17}
\]
Therefore,
\[
\w_{8}\left(M_{\mu}^{1}\right) = \min_{n\in\Z\setminus\{0\}} 
\left\{
c_{\vol}=\sqrt{\frac{2\mu+1}{8}},~u_{1}(\mu,n),~u_{2}(\mu, n),~u_{3}(\mu,n),~u_{3}(\mu,0)=\frac{6\mu+6}{17}
\right\}
\]
which proves Theorem~\ref{thm:MainTheoremTwistedK=8}. 

In order to describe $\w_{8}(\mu)$ explicitely as a piecewise linear function, we introduce the  functions
\[
s_{1}(n) = \frac{4n(3n-2)}{24n^{2}+8n+1}, \quad
s_{2}(n) = \frac{4n(3n+2)}{24n^{2}+40n+17}, \quad
s_{3}(n) = \frac{8n^{2}+8n+1}{16(n+1)^{2}}
\]
defined respectively on $\Z$, $\Z$, and $\Z\setminus\{-1\}$. For convenience, we  extend the domain of $s_{3}(n)$ to $\Z$ by setting $s_{3}(-1) = \infty$. Simple but rather tedious computations show that 
\[u_{1}(\mu,n) = u_{2}(\mu, n) \iff \mu = s_{1}(n)\]
\[u_{2}(\mu,n) = u_{3}(\mu, n) \iff \mu = s_{2}(n)\]
\[u_{3}(\mu,n) = u_{1}(\mu, n+1) \iff \mu = s_{3}(n)\]
and that
\[u_{1}(\mu,n) = u_{2}(\mu,n) < c_{\vol} \text{~for~} \mu=s_{1}(n)\]
\[u_{2}(\mu,n) = u_{3}(\mu,n) < c_{\vol} \text{~for~} \mu=s_{2}(n)\]
\[u_{3}(\mu,n) = u_{1}(\mu,n+1) = c_{\vol} \text{~for~} \mu=s_{3}(n)\]
Moreover, for $n\geq 0$, the $s_{i}(n)$ form interlocking increasing sequences which converge to $1/2$ as $n\to\infty$,
\[0=s_{2}(0)<\frac{1}{16}=s_{3}(0)<\cdots < s_{1}(n)<s_{2}(n)<s_{3}(n)<s_{1}(n+1)<\cdots<1/2\]
while for negative $n\leq -1$, the $s_{i}(n)$ form interlocking decreasing sequences which also converge to $1/2$ as $n\to-\infty$,
\[1/2<\cdots< s_{3}(n-1)  <s_{1}(n)<s_{2}(n)<s_{3}(n)<\cdots< 4 = s_{2}(-1)<\infty = s_{3}(-1) \]
We conclude that on the interval $(0,~4]$, the upper bounds $u_{i}(\mu, n)$ are never greater than $c_{\vol}$. On the other hand, by Lemma~\ref{lemma:TwistedVp-2ordered}, the conclusions of Proposition~\ref{prop:TwistedEvenGeneric} hold whenever $\mu\geq \frac{7}{4}$. In particular, we know that $\w_{8}(\mu) = 1 = u_{3}(\mu,-1)$ whenever $\mu\geq 4$. Together with the fact that $\w_{8}(1/2) = 1/2$, this gives a complete description of $w_{8}\left(M_{\mu}^{1}\right)$, namely
\printlabel{thm:W8PiecewiseTwisted}
\begin{thm}\label{thm:W8PiecewiseTwisted}
The generalized Gromov width $\w_{8}(M_{\mu}^{1})$ is the piecewise linear function defined~by
\[
\w_{8}(M_{\mu}^{1})=
\begin{cases}
 u_{2}(\mu,n) & \text{~if~} \mu\in\left( s_{1}(n),~ s_{2}(n) \right]\\
 u_{1}(\mu,n) & \text{~if~} \mu\in\left( s_{3}(n-1),~ s_{1}(n) \right]\\
 u_{3}(\mu,n-1) & \text{~if~} \mu\in\left( s_{2}(n-1),~ s_{3}(n-1) \right)\\
 c_{\vol}=\sqrt{\frac{2\mu+1}{8}} & \text{~if~} \mu \in 
 \left\{s_{3}(n-1),~\frac{1}{2},~s_{3}(-(n+1))\right\},\qquad \text{where $n\geq 1$}\\
 u_{1}(\mu,-n) & \text{~if~} \mu\in\left(s_{3}(-(n+1)),~ s_{1}(-n)\right]\\
 u_{2}(\mu,-n) & \text{~if~} \mu\in\left(s_{1}(-n),~ s_{2}(-n)\right]\\
 u_{3}(\mu,-n) & \text{~if~} \mu\in\left(s_{2}(-n),~ s_{3}(-n)\right)\\
\end{cases}
\] 
\end{thm}
\cqfd
Let define the set $\mathcal{S}\subset (0,\infty)$ by setting
\[
\mathcal{S} = \left\{s_{3}(n-1),~\frac{1}{2},~s_{3}(-(n+1))\right\}
 = \left\{\frac{8n^{2}-8n+1}{16n^{2}},~\frac{1}{2},~\frac{8n^{2} + 8n + 1}{16n^{2}}\right\},~n\geq 1
\]
Note that $\mathcal{S}\subset (0,~ 17/16]$. Theorem~\ref{thm:W8PiecewiseTwisted} shows that $M_{\mu}^{1}$ admits a full packing by $8$ balls if, and only if, $\mu\in\mathcal{S}$. This last result allows us to complete our computations of the stability numbers of $M_{\mu}^{1}$.
\printlabel{cor:EvenStabilityNumbersTwisted}
\begin{cor}\label{cor:EvenStabilityNumbersTwisted}
The even stability number of $M_{\mu}^{1}$ is
\[
N_{{\rm even}}\left(M_{\mu}^{1}\right)=
\begin{cases}
8 & \text{~if~}\mu\in\mathcal{S}\\
10 & \text{~if~} \mu\in\left(0,\frac{3}{2}\right]\setminus \mathcal{S}\\
2\left\lceil \mu+1+\sqrt{2\mu+1}\right\rceil 
			& \text{~if~}\mu\in\left[\frac{3}{2},~\infty\right)
\end{cases}
\]
\end{cor}
\begin{proof}
By Proposition~\ref{prop:PackingNumbersTwisted<8}, the only pair $(\mu,k)$ for which we have full packings by $k=2p\leq 6$ balls is $(1/4,~6)$, and we observe that $1/4\not\in\mathcal{S}$. On the other hand, for $\mu\in(0,1/2]$, Lemma~\ref{lemma:TwistedWkgeq9} shows that we have full packings whenever $k\geq 9$. We conclude that for $\mu\in(0,1/2]\setminus\mathcal{S}$ the even stability number is equal to $10$, while it is equal to $8$ for $\mu\in\mathcal{S}\cap(0,1/2]$. 

When $\mu\in(1/2, \infty)$, the largest root of the polynomial in $p$
\[
2p\left(p^{2} - 2(\mu+1)p + \mu^{2}\right)
\]
obtained by setting
\[c_{\vol}^{2}=\frac{(p+\mu)^{2}}{(2p)^{2}}\]
is
\[r(\mu) =  \mu+1+\sqrt{2\mu+1}\]
The integer $J(\mu) := \max \{10,~ 2\lceil r(\mu) \rceil\}$ gives the even stability number in the range $k\geq 10$. It is easy to see that on $\mathcal{S}\cap(1/2,~ 3/2]$, $J(\mu)=10$. The results follows readily.
\end{proof}
Combining Corollary~\ref{cor:OddStabilityNumbersTwisted} with Corollary~\ref{cor:EvenStabilityNumbersTwisted}, we finally get the general stability number of the twisted bundle, namely
\begin{cor}
The general stability number of $M_{\mu}^{1}$ is
\[
N_{{\rm stab}}\left(M_{\mu}^{1}\right)=
\begin{cases}
8 & \text{~if~} \mu\in  \mathcal{S} \\
9 & \text{~if~} \mu\in   \left(0,~\frac{3}{2}\right) \\
N_{{\rm even}}-1  & \text{~if~} \mu\in \left[\frac{3}{2},~\infty \right)
\end{cases}
\]
\end{cor}
\cqfd

\section{Embedding ellipsoids in polydisks and comparison with ECH capacities}

\subsection{Embedding ellipsoids in polydisks}
Using a recent result of D. Muller our results about ball packings can be translated into the following
\begin{cor}\label{muler2} 
Let $k $ be any integer greater than $8$ and let $a,s,t$ be any positive real numbers with $s < t.$ Set $\mu = a/s$. Then the following statements are equivalent:
\renewcommand{\labelenumi}{\roman{enumi})}
\renewcommand{\theenumi}{\roman{enumi})}
\begin{enumerate}
  \item $E(a, k a) \hookrightarrow P(s,t)$ 
  \item If $k =2p+1 $ then 
\[
\frac{a}{s} \leq
\begin{cases}
c_{\vol}=\sqrt{\frac{ 2 \mu}{2p+1}} & \text{~if~}\mu\in\left[ 1,~ p+1-\sqrt{2p+1} \right)   \\
\frac{\mu+p}{2p+1}         & \text{~if~}\mu\in\left[ p+1-\sqrt{2p+1},~ p+1 \right) \\
1                          & \text{~if~}\mu\in\left[ p+1,~ \infty \right)
\end{cases}
\]
If $k=2p$ then 
\[
\frac{a}{s} \leq
\begin{cases}
c_{\vol}=\sqrt{\frac{\mu}{p}}  & \text{~if~}\mu\in\left[ 1,~ \frac{p-2 + \sqrt{p^2 -4p}}{2}\right)\\
\frac{a_{n-1}\mu + a_n}{2 (a_n+a_{n-1}) - 1} & \text{~if~}\mu\in\left[ \frac{\gamma_{n}}{\gamma_{n-1}},~ \frac{\gamma_{n-1}}{\gamma_{n-2}} \right) ,~n\geq 2\\
1 & \text{~if~}\mu\in\left[ p,~ \infty \right)
\end{cases}
\]
\end{enumerate}
\end{cor}
\begin{proof}
For a generic almost complex structure $J$ on $M_{\mu}^0$, $\mu=t/s$, we can arrange that the image of an embedding of $k$ disjoint equal balls into $M_{\mu}^0$ misses some generic section of $S^2 \times S^2$ in the homology class $[S^{2}\times\{*\}]$. Hence, we can view such an embedding as an embedding into the polydisk $P(1, \mu)$. 
In~\cite{M}, D. McDuff proved that the existence of an embedding of $k$ balls of equal sizes $c$ in the ball $B^{4}(1)$ is equivalent to the existence of a symplectic embedding of an ellipsoid $E(c,k c)$ in the same ball. Recently D. Muller~\cite{Mu} used similar ideas to prove an analog result for the embeddings of ellipsoids into polydisks, see for instance~Proposition $10$ in~\cite{H4}. According to her results, if one has an symplectic embedding
\begin{equation}
  \Phi: \sqcup_{k} {\rm int} B(\w_{k}) \cup B(1) \cup B(\mu) \longrightarrow B(1+ \mu)
\end{equation}
then one obtains an embedding of $E(\w_{k},~ k\cdot\w_{k}) \rightarrow P(1,~\mu)$.
But it is clear that the problem of finding such embedding $\Phi$ reduces to proving that the vector $v_{0}=(\mu+1-c\,;\mu+1,c^{\times (k-1)}, 1-c)$ belongs to the symplectic cone of $X_{k+1}$. Hence the equivalence between (i) and~(ii) follows. 
\end{proof}

\subsection{Comparison with ECH capacities}

In a recent series of papers, M. Hutching's defines the embedded contact homology (ECH) capacities for Liouville domains $(Y, \xi)$ and, more generally, for Liouville domains with corners. The purpose of this section is to establish a connection between our results and  ECH capacities of ellipsoids and polydisks. Let us first give a brief overview of the necessary notations and  results existing in the literature.  The ECH capacities form a sequence $c_k(Y, \xi)$ which represents the spectrum of a filtered version of embedded contact homology, in which the filtering is defined using a certain action functional. The construction of this homology theory, as well as its mains properties, are discussed in Hutchings~\cite{H2} and~\cite{H4}. We will consider here the case of a Liouville domain given by an ellipsoid $E(a,b)$ and that of a Liouville domain with corners given as a polydisk $P(s,t)$. We will denote by $\mathcal{N}_k(a,b)$ the sequence of ECH capacities $c_k(E(a,b))$, and by $\mathcal{M}_{k}(s,t)$ the sequence $c_k(P(s,t))$. For our purpose, it is sufficient to recall that the following results:
\begin{thm}[see M. Hutchings~\cite{H2}, ~\cite{H4}] \label{thm:Hutchings}
\renewcommand{\labelenumi}{\roman{enumi})}
\renewcommand{\theenumi}{\roman{enumi})}
\begin{enumerate}
  \item $E(a, b) \hookrightarrow P(s,t)$ if, and only if $\mathcal N(a,b) <
    \mathcal M (s, t)$
  \item For an ellipsoid $E(a,b)$, the elements of the sequence
    $\mathcal { N}(a,b)$ are obtained by arranging in increasing order
    (with repetitions) all the numbers of the type $am+bn$ with $m,n$
    natural numbers.
  \item For a polydisk $P(s, t)$, the ECH capacities are organized in
    a sequence $\mathcal{M}(s,t)$ whose $i^{\text{th}}$ element is
    defined as
    \begin{equation}\label{polydisc}
      {\mathcal M}_i(\nu,\mu)= \min\{\nu m+ \mu n ~|~ (m+1)(n+1) \geq
      i+1, (m,n) \in \N \times \N \}.
    \end{equation} 
\end{enumerate}
\end{thm}
Note that the reverse implication in statement $(i)$ of this theorem is a consequence of how the invariants are defined by Hutchings. The direct implication was recently proved by M. Hutchings in~\cite{H4} using D. Muller's \cite{Mu} result cited above, as well as a strategy provided by Mcduff in \cite{mchof} which proves sharpness of the ECH invariants for embeddings of ellipsoids into ellipsoids.

The computations of the generalized Gormov widths provides explicit and comprehensive ranges of parameters $a,b,s,t$ for which such embeddings exists in the case the ratio $a/b$ is an integer greater than $8$. Our Corollary~\ref{muler} gives an alternative proof of the direct implication in (i) in Theorem~\ref{thm:Hutchings} under the same integral condition. (For smaller values of $k$ one can trace the results using the Appendix A). This alternative proof sheds some insight on the difficulties and intricacies involved in computing explicitly the ECH invariants. We note that such computation is needed if one wants to find optimal values $a,,b,s,t$ for which embeddings $E(a, b) \hookrightarrow P(s,t)$ exist, without making use of the reduction algorithm.

\begin{prop}\label{ech} Let $k \geq 8$ be an integer. The following are equivalent:
\renewcommand{\labelenumi}{\roman{enumi})}
\renewcommand{\theenumi}{\roman{enumi})}
\begin{enumerate}
  \item $E(a, k a) \hookrightarrow P(s,t)$ 
  \item $\mathcal N(a,k a) < \mathcal M (s, t)$
  \item $\frac {a}{s}  \mathcal N(1,k) < \mathcal M (1, \frac {t}{s})$
  \item If $k =2p+1 $ then $\frac{a}{s} \leq \w_k =\min\{1, c_{\vol},
    \frac{\mu + p}{2p+1}\}.$
 If $k=2p$ then $\frac {a}{s} \leq \w_{k}
    =\min_{n \in \N} \{1, c_{\vol}, w_n \}.$ 
Moreover, Theorem~\ref{thm:MainTheoremTrivial} the precise value of this minimum is given by the index $n$ of the interval $I_n$ in which $\mu=\frac {t}{s}$ lies.
\end{enumerate}
\end{prop}

\begin{proof}
Note that $(i) \implies (ii)$ is the inverse implication  of point $i)$ from Hutching's theorem~\ref{thm:Hutchings}, and that $(ii) \implies(iii)$ is straightforward as all ECH capacities satisfy a rescaling property. The implication $(iv)\implies (i)$ is covered by Corrolary~\ref{muler}.

We will show here the remaining needed implication, namely that $(iii)
$ implies $(iv)$. So let us assume that $\frac {a}{s} \mathcal N(1,k)
< \mathcal M (1, \mu)$.  The fact that $\frac {a}{s} \leq
\min\{1,c_{\vol}\}$ is straightforward as the ECH capacities respect
volume and because the first entries of $\mathcal N(1,k)$ and
$\mathcal M (1, \frac {t}{s})$ are 1.  Let us prove the rest of the
inequalities.

Let us consider the case $k=2p$. Recall that for any $n >1$, $w_n= \frac{a_n +a_{n-1}\mu}{2(a_n+a_{n-1})-1}$.  We will introduce the sequence $x_n$ satisfying the identity $2(a_n+ a_{n-1}) -1 = 2p x_n + (-1)^{n+1}$. One can easily verify, using the recurrence~\eqref{MainRecurrence} for the sequence $a_n$, that the numbers $x_n$ are in fact natural numbers satisfying the relation:
\[x_{n+3} = (p-1)(x_{n+2}-x_{n+1}) + x_n + (-1)^{n+1} \]
Therefore 
\begin{equation}\label{relxa}
  x_n = \frac {2(a_n +a_{n-1}) -1+(-1)^{n}} {2p}
\end{equation}
For each $n>1$ we define the index $i$ to be
\begin{equation}\label{theindex}
  i_n:= (a_n +1) (a_{n-1} +1) -1
\end{equation}
Then it is clear from~\eqref{polydisc} that 
\[M_{i_n}(1, \mu) = a_n + a_{n-1} \mu.\]
Therefore our assumption is equivalent with
\begin{equation}\label{thewECH}
  \frac{a}{s} \leq \frac{ a_n + a_{n-1} \mu}{N_{i_n}(1,k)}
\end{equation}
We claim that for our choice of $i_n$ , we get that
\begin{equation}\label{identityholds}
  N_{i_n}(1,k) = 2p x_n + (-1)^{n+1}= 2(a_n+ a_{n-1}) - 1,
\end{equation}
hence the right hand side of~\eqref{thewECH} is exactly $\w_n$. The
remaining of the proof will be to justify the value of the
$i_n^{\text{th}}$ ECH capacity of $E(1,k)$ from relation~\eqref{identityholds}. To see this, first observe that for any
  integer $k$, $\mathcal{N}(1, k)$ is given by
\begin{multline}\label{ellipse}
  \Big( 1, \ldots k-1, k, k, k+1, k+1, \ldots 2k-1, 2i-1, (2k)^{\times 3}, \ldots  \\
  \ldots (3k-1)^{\times 3}, \ldots ,  (j k)^{\times (j+1)}, \ldots , ((j+1)k-1)^{\times (j+1)}, \ldots, \Big)
\end{multline}
In particular, any number of the form $k x-1 $ will appear as a value
of $N_i(1,k)$ exactly when
\begin{equation}\label{xeven}
  k x(x+1) /2 -x \leq i \leq kx (x+1)/2 -1
\end{equation}
and any number of the form $k x+1 $ will appear as a value of
$N_i(1,k)$ exactly when
\begin{equation}\label{xodd}
  k x(x+1) /2 +x +1 \leq i \leq kx (x+1)/2 +2x+2.
\end{equation}
The equation~\eqref{identityholds} will then follow from the following
claim used in conjunction with~\eqref{xeven} when $n$ is even and with~\eqref{xodd} when $n$ is odd:
\[2p x_n(x_n+1) + (-1)^n x_n + \frac {1+ (-1)^n}{2} = i_n\]
To prove this identity we first observe that it is equivalent, via the
identities~\eqref{relxa} and~\eqref{theindex}, with the identity
\begin{equation}
  \frac{(a_n+a_{n-1})^2}{p} - \frac{a_n+a_{n-1}}{p} = a_n a_{n-1}.
\end{equation}
But this was proved in Lemma~\ref{lemma:numerical}.

Let us now consider the case when $k=2p+1$. In this case we pick the index $i=2p+1$. We get that $\mathcal { N}_{2 p +1}(1, k)= (2 p+1) $ by~\eqref{xodd}. On the other hand, for $k=2 p +1$ the condition in the equation~\eqref{polydisc} is satisfied with equality if $(m,n)=(p,1)$ and it implies that $\mathcal { M}_{2 p+1}(1,\mu) =\mu+p$. Hence the inequality
$\frac {a}{s}\mathcal{N}_{2p+1}(1,k) \leq \mathcal{M}_{2p+1}(1,\mu)$ is equivalent with
\[\frac{a}{s} \leq \frac{\mu +p}{2p-1}\]
and the result follows. This concludes the proof.
\end{proof}
\begin{remark}
Note that one can think of this Proposition as one step forward towards proving Corollary~\ref{muler} (thus the equivalence $(i)<=>(iv)$ ) without our results regarding the reduction algorithm, by making use instead of    {\it both} implications available from Theorem~\ref{thm:Hutchings}. Indeed, one could conjure the numbers $N_{i_n}(1,k) $ and $M_{i_n}(1,\mu) $ (albeit we believe it difficult without the previous  knowledge on all recurrences and results  obtained form the algorithm) and obtain the implication  $(i) =>(iv)$. But the reverse of this implication requires that one shows that the {\it entire} vector  $\frac {a}{s} \mathcal N(1,k)
< \mathcal M (1, \mu)$ for the proposed values for  and that would mean computing all the ECH capacities for the two objects. But our main results  does, in addition providing insight to what particular values should we pick for $a/s$, circumvent an attempt to compute all values of the entries in the ECH capacities, by showing  that the embeddings exist for the proposed values $a/f= \w_k(\mu)$. 
\end{remark}

We will conclude, therefore, with the following two questions:
\begin{itemize}
\item Can we give a simpler proof of Theorem~\ref{thm:MainTheoremTrivial} using exclusively the computation of ECH capacities introduced by M. Hutchings~\ref{thm:Hutchings}\,?
\item Is there a way to reduce the computations of the generalized Gromov widths of the twisted bundle $M^1_{\mu}$ to a comparison of suitable sequences of ECH capacities\,?
\end{itemize}

\appendix
\section{Embeddings of $1\leq k\leq 7$ disjoint balls in $M_{\mu}^{i}$}

For $1\leq k\leq 7$, the set $\E_{K}\subset H_{2}(X_{k+1}; \Z)$ of exceptional homology classes of the $(k+1)$-fold blow-up of $\CP^{2}$ is finite. It consists of classes of the following types:
\[
\left(0\,;-1\right), \quad
\left(1\,;1^{\times 2}\right), \quad
\left(2\,;1^{\times 5}\right)\]
\[
\left(3\,;2^{\times 1},1^{\times 6}\right), \quad
\left(4\,; 2^{\times 3}, 1^{\times 5} \right), \quad 
\left(5\,; 2^{\times 6}, 1^{\times 2}\right), \quad
\left(6\,; 3, 2^{\times 7}\right)
\]
It follows that the symplectic cone $\mathcal{C}_{K}$ of $X_{k+1}$ is defined by finitely many inequalities. In particular, an easy computation yields the packing numbers $p_{k}(M_{\mu}^{i})$, $1\leq k\leq 7$. Using the same normalization as before, we obtain:

\printlabel{prop:PackingNumbersProduct<8}
\begin{prop}\label{prop:PackingNumbersProduct<8}
For the normalized product bundle $M_{\mu}^{0}=\left(S^{2}\times
S^{2}, \mu\sigma\oplus\sigma\right)$, the packing numbers
$p_{k}(M_{\mu}^{0})$, $1\leq k\leq 7$ are given by
\begin{alignat*}{2}
p_{1}(M_{\mu}^{0}) &= 
\frac{1}{2\mu} &
\quad p_{2}(M_{\mu}^{0}) &= 
\frac{1}{\mu}\\
p_{3}(M_{\mu}^{0}) &= 
\begin{cases}
\frac{3}{2\mu} \big( \frac{\mu+1}{3}\big)^{2} & \text{if $\mu\in[1,2)$} \\
\frac{3}{2\mu} & \text{if $\mu\in[2,\infty)$}
\end{cases} &
p_{4}(M_{\mu}^{0}) &=
\begin{cases}
\frac{2}{\mu} \big( \frac{\mu+1}{3}\big)^{2} & \text{if $\mu\in[1,2)$} \\
\frac{2}{\mu} & \text{if $\mu\in[2,\infty)$}
\end{cases}\\
p_{5}(M_{\mu}^{0}) &= 
\begin{cases}
\frac{5}{2\mu} \big( \frac{\mu+2}{5}\big)^{2} & \text{if $\mu\in[1,3)$} \\
\frac{5}{2\mu} & \text{if $\mu\in[3,\infty)$}
\end{cases} &
p_{6}(M_{\mu}^{0}) &= 
\begin{cases}
\frac{3}{\mu}\big(\frac{2\mu+2}{7}\big)^{2}& \text{if $\mu\in[1,\frac{4}{3})$}\\
\frac{3}{\mu}\big(\frac{\mu+2}{5}\big)^{2}& \text{if $\mu\in[\frac{4}{3},3)$}\\
\frac{3}{\mu} & \text{if $\mu\in[3,\infty)$}
\end{cases}\\
\end{alignat*}
\[p_{7}(M_{\mu}^{0}) =
\begin{cases}
\frac{7}{2\mu}\big(\frac{4\mu+4}{15}\big)^{2}&\text{if $\mu\in[1,\frac{8}{7})$}\\
\frac{7}{2\mu}\big(\frac{3\mu+4}{13}\big)^{2}& \text{if $\mu\in[\frac{8}{7},\frac{11}{8})$}\\
\frac{7}{2\mu}\big(\frac{\mu+3}{7}\big)^{2}&\text{if $\mu\in[\frac{11}{8},4)$}\\
\frac{7}{2\mu} & \text{if $\mu\in[4,\infty)$}
\end{cases}\]
In particular, the pairs $(\mu, k)$, $1\leq k\leq 7$, for which we
have full packings of $M_{\mu}^{0}$ are:
\[
\left\{ (1,2), (2,4), (\frac{4}{3}, 6), (3,6), (\frac{8}{7},7) \right\}
\]
\end{prop}

\printlabel{prop:PackingNumbersTwisted<8}
\begin{prop}\label{prop:PackingNumbersTwisted<8}
For the normalized twisted bundle $M_{\mu}^{1}$, the packing numbers
$p_{k}(M_{\mu}^{1})$, $1\leq k\leq 7$ are given by
\begin{alignat*}{2}
p_{1}(M_{\mu}^{1}) &=
\frac{1}{2\mu+1} & 
\quad p_{2}(M_{\mu}^{1}) &= 
\begin{cases}
\frac{2}{2\mu+1} \big( \frac{\mu+1}{2}\big)^{2} & \text{if $\mu\in(0,1)$} \\
\frac{2}{2\mu+1} & \text{if $\mu\in[1,\infty)$}
\end{cases}\\
p_{3}(M_{\mu}^{1}) &= 
\begin{cases}
\frac{3}{2\mu+1} \big( \frac{\mu+1}{2}\big)^{2} & \text{if $\mu\in(0,1)$} \\
\frac{3}{2\mu+1} & \text{if $\mu\in[1,\infty)$}
\end{cases}
&\quad p_{4}(M_{\mu}^{1}) &=
\begin{cases}
\frac{4}{2\mu+1}\big(\frac{\mu+2}{4}\big)^{2}&\text{if $\mu\in(0,2)$}\\
\frac{4}{2\mu+1} & \text{if $\mu\in[2,\infty)$}
\end{cases}\\
p_{5}(M_{\mu}^{1}) &= 
\begin{cases}
\frac{5}{2\mu+1} \big( \frac{2\mu+2}{5}\big)^{2} & \text{if $\mu\in(0,\frac{2}{3}]$} \\
\frac{5}{2\mu+1} \big( \frac{\mu+2}{4}\big)^{2} & \text{if $\mu\in(\frac{2}{3}, 2]$} \\
\frac{5}{2\mu+1} & \text{if $\mu\in(2,\infty)$}
\end{cases} &
\quad p_{6}(M_{\mu}^{1}) &= 
\begin{cases}
\frac{6}{2\mu+1}\big(\frac{2\mu+2}{5}\big)^{2}& \text{if $\mu\in(0,\frac{1}{4}]$}\\
\frac{6}{2\mu+1}\big(\frac{2\mu+3}{7}\big)^{2}& \text{if $\mu\in(\frac{1}{4},\frac{3}{5}]$}\\
\frac{6}{2\mu+1}\big(\frac{\mu+3}{6}\big)^{2}& \text{if $\mu\in(\frac{3}{5},3]$}\\
\frac{6}{2\mu+1} & \text{if $\mu\in(3,\infty)$}
\end{cases}\\
\end{alignat*}
\begin{equation*}
p_{7}(M_{\mu}^{1}) =
\begin{cases}
\frac{7}{2\mu+1}\big(\frac{3\mu+3}{8}\big)^{2}&\text{if $\mu\in(0,\frac{1}{7}]$}\\
\frac{7}{2\mu+1}\big(\frac{4\mu+5}{13}\big)^{2}& \text{if $\mu\in(\frac{1}{7},\frac{3}{8}]$}\\
\frac{7}{2\mu+1}\big(\frac{4\mu+6}{15}\big)^{2}&\text{if $\mu\in(\frac{3}{8},\frac{6}{11}]$}\\
\frac{7}{2\mu+1}\big(\frac{3\mu+6}{14}\big)^{2}&\text{if $\mu\in(\frac{6}{11},\frac{3}{2}]$}\\
\frac{7}{2\mu+1}\big(\frac{\mu+3}{6}\big)^{2}&\text{if $\mu\in(\frac{3}{2},3]$}\\
\frac{7}{2\mu+1} & \text{if $\mu\in(3,\infty)$}
\end{cases} \\
\end{equation*}
In particular, the pairs $(\mu, k)$, $1\leq k\leq 7$, for which we have full packings of $M_{\mu}^{1}$ are:
\[
\left\{(1,3), (\frac{1}{4}, 6), (\frac{1}{7}, 7), (\frac{3}{8},7), (3,7)\right\}
\]
\end{prop}


\end{document}